\documentclass[a4paper,12pt]{amsart}
\usepackage[T1]{fontenc}
\usepackage[utf8]{inputenc}
\usepackage[english]{babel}
\usepackage{etoolbox}
\usepackage{amsmath, amssymb, amsthm, amsfonts}
\usepackage{mathrsfs}
\usepackage{tikz}
\usetikzlibrary{arrows}
\usetikzlibrary{matrix}
\usepackage[left=2.2cm, right=2.2cm, top=2cm]{geometry}
\usepackage{enumerate}
\usepackage{hyperref}
\date{}
\usepackage[font=small,labelfont=bf]{caption}
\theoremstyle{plain}
\newtheorem{thm}{Theorem} 
\newtheorem{lem}{Lemma} 
\newtheorem{cor}{Corollary} 

\newtheorem{prop}{Proposition} 

\nonstopmode\numberwithin{equation}{section}
\theoremstyle{definition}
\newtheorem{defi}{Definition} 
\newtheorem{rema}{Remark}

\newtheorem*{ques*}{Question}

\makeatletter
\def\tagform@#1{\maketag@@@{\ignorespaces#1\unskip\@@italiccorr}}
\let\orgtheequation\theequation
\def\theequation{(\orgtheequation)}
\makeatother
\let\orgautoref\autoref

\renewcommand{\autoref}[1]{\def\equationautorefname{}\orgautoref{#1}}

\usepackage{fancyhdr}

\newcommand\shorttitle{Stochastic Bouligand Landweber}

\begin{document}

\title[Stochastic Bouligand Landweber iteration]{Stochastic Data-Driven Bouligand Landweber Method for Solving Non-smooth Inverse Problems}
\author{ Harshit Bajpai$^{\dagger}$, Gaurav Mittal$^{\ddagger}$, Ankik Kumar Giri$^{\dagger}$}
\email{bajpaiharshit87@gmail.com, gaurav.mittaltwins@yahoo.com, ankik.giri@ma.iitr.ac.in}
\address{$^\dagger$Department of Mathematics, Indian Institute of Technology Roorkee, Roorkee, Uttarakhand, 247667, India}
\address{$^\ddagger$Defence Research and Development Organization, Near Metcalfe House, Delhi,  110054, India}
\maketitle
\begin{abstract} 
   In this study, we present and analyze a novel variant of the stochastic gradient descent method, referred as Stochastic data-driven Bouligand Landweber iteration tailored for addressing the system of non-smooth ill-posed inverse problems. Our method incorporates the utilization of training data, using  a  bounded linear operator, which guides the iterative procedure. At each iteration step, the method randomly chooses one equation from the nonlinear system with data-driven term. When dealing with the precise or exact data, it has been established that mean square iteration error converges to zero. However, when confronted with the noisy data, we employ our approach in conjunction with a predefined stopping criterion, which we refer to as an \textit{a-priori} stopping rule.
We provide a comprehensive theoretical foundation, establishing convergence and stability for this scheme within the realm of infinite-dimensional Hilbert spaces. These theoretical underpinnings are further bolstered by discussing an example that fulfills assumptions of the paper. 
 \end{abstract}

\vspace{.3cm}

\noindent

\hspace{-56mm}\subjclass {}{\textbf{AMS Subject Classifications}: 47H17, 65J15, 65J20}\\
 \section{Introduction}\noindent
 This work is about deducing the approximate solution of the system of nonlinear ill-posed equations of the form
\begin{equation}\label{ststart}
    F_{i}(u) = y^{\dagger}_i, \hspace{2mm} i = 0, 1,\ldots, P-1,
\end{equation}
where,  $F_i : \mathbb{D}(F_i) \subset U \rightarrow Y$, for each $i$,  represents a nonlinear operator that may not be G\^ateaux differentiable  between the Hilbert spaces $U$ and $Y$. Here $P$ is a positive integer,  $y_{i}^{\dagger} \in Y$ represents the precise or exact data, $\mathbb{D}(F_i)$ denotes the domain of the operator $F_i$ and the Hilbert spaces $U$ and $Y$ are equipped with  usual inner products $(\cdot,\cdot)$ and norms $\|\cdot\|$, respectively.   Alternatively, for the product space $Y^P=Y\times Y\times\cdots Y$ ($P$ copies),  \ref{ststart} can be reformulated as 
\begin{equation}\label{stcombined}
  F(u) = y^{\dagger} , 
\end{equation}
where $F : U \rightarrow Y^P$ is defined as 
$$ F(u) = \begin{pmatrix} 
           F_0(u) \\
           \vdots \\
           F_{P-1}(u)
          \end{pmatrix}\ \  \text{and}\ \hspace{2mm} 
         y^{\dagger}= \begin{pmatrix} 
           y^{\dagger}_0 \\
           \vdots \\
           y^{\dagger}_{P-1}
          \end{pmatrix}. $$
As usual, instead of the exact data, we assume the availability of noisy data $y^{\delta}$ such that 
$$\|y^{\dagger} - y^{\delta} \| \leq \delta,$$
for noise level $\delta \geq 0$. Here, we write $\delta = (\delta_0, \delta_1,... ,\delta_{P-1})$ and $y^{\delta}= (y^{\delta}_0, y^{\delta}_1,...,
y^{\delta}_{P-1})$ and these fulfills
\begin{equation}
    \|y_{i}^{\delta} - y_i^{\dagger}\| \leq \delta_i \hspace{5mm} i = 0, 1,...,P-1.
\end{equation}
 Due to ill-posedness of $\ref{ststart}$, its solution may not exist and even if it exists,   it may not be unique. Furthermore, the solution(s) may be unstable with respect to the noisy data $y^{\delta}$ (see \cite{engl1996regularization}). Consequently, regularization methods are needed  for deducing the stable approximate solution of \ref{ststart}. \newline
 One of the prominent and highly effective classes of regularization methods is that of iterative regularization.  This regularization class  has found a tremendous success in addressing a wide range of inverse problems (see \cite{engl1996regularization, kaltenbacher2008iterative} and the references therein). Further, one of the well-studied classical iterative methods is known as \emph{Landweber iteration method} (LIM) and for a Fr\'echet differentiable (or smooth) forward mapping $F: U \rightarrow Y$, LIM  can be written as (see \cite{engl1996regularization, kaltenbacher2008iterative})
 \begin{equation}\label{LB}
     u_{k+1}^{\delta} = u_{k}^{\delta} - F'(u_k^{\delta})^*(F(u_k^{\delta}) - y^{\delta}), \hspace{10mm} k \geq 0.
 \end{equation}
 Here $u_0^{\delta}:=u^{(0) }$ is a first guess that takes into account the available knowledge about the solution to be recovered and $F'(u_k^{\delta})^*$ is the adjoint of Fr\'echet derivative of $F$ at $u_k^{\delta}$. When the noisy data is available, a proper stopping rule (a-posteriori) must be employed in order to show that the iterative scheme  \ref{LB}  is a regularization one (see \cite{kaltenbacher2008iterative}). The commonly incorporated stopping rule is well known as \emph{discrepancy principle}, i.e.,
 the method is halted after   $\Ddot{k} = \Ddot{k}(\delta, y^{\delta})$ steps, where 
 \begin{equation}\label{discrepancy}
     \|F(u_{\Ddot{k}}^{\delta}) - y^{\delta}\| \leq \tau \delta <  \|F(u_{k}^{\delta}) - y^{\delta}\|, \hspace{10mm} 0 \leq k < \Ddot{k}, 
 \end{equation}
 for some $\tau >1.$ Scherzer modified the method  \ref{LB} in \cite{scherzer1998modified} by introducing a damping term and named it as \emph{iteratively reguarlized Landweber iteration method}. This method can be written as
      \begin{equation}\label{damping}
     u_{k+1}^{\delta} = u_{k}^{\delta} - F'(u_k^{\delta})^*(F(u_k^{\delta}) - y^{\delta}) - \lambda_{k}(u_k^{\delta} - u^{(0)}), \hspace{10mm} k \geq 0,
 \end{equation}
 where  $\lambda_k\geq 0$. In comparison to the method \ref{LB}, the convergence rates analysis of the method  \ref{damping}  
   needs fewer assumptions on the mapping $F$. Furthermore, the method \ref{damping} converges to a solution which is close to $u^{(0)}$, however, the method \ref{LB} needs additional assumptions for the similar convergence behavior  (see \cite{aspri2020data, kaltenbacher2008iterative}). This means that the inclusion of additional damping term in \ref{damping} has several advantages over the method \ref{LB}.
 
 Motivated by these observations, Aspri et al. in \cite{aspri2020data}  used a black box strategy to introduce an a-priori data driven term in  \ref{LB}  for the inclusion of  the image data $\left(F(u^{(l)})_{1\leq l \leq N}\right)$, where $N$ is a positive integer. The method of Aspri et al. can be written as 
\begin{equation}\label{Data-driven}
     u_{k+1}^{\delta} = u_{k}^{\delta} - F'(u_k^{\delta})^*(F(u_k^{\delta}) - y^{\delta}) -  \lambda_{k}^{\delta}M'(u_k^{\delta})^*(M(u_k^{\delta}) - y^{\delta}), \hspace{10mm} k \geq 0.
 \end{equation}
Here  $M$ maps each $u^{(l)}$ to $F(u^{(l)})$ and vice-versa for each $l=1, 2, \cdots, N$, $\lambda_k^{\delta}\geq 0$ and $M$ is Fr\'echet differentiable.  The strong convergence results of the method \ref{Data-driven}  along with the stability of the method were discussed  under certain assumptions in \cite{aspri2020data}. Very recently, Tong et al. \cite{tong2023data} proposed a data driven Kaczmarz type iterative regularization method with uniformly convex constraints. This method has a notable acceleration effect in comparison to the method \ref{Data-driven}. \\
It is clear from the formulations of the methods \ref{LB}, \ref{damping} and \ref{Data-driven}   that these are not applicable on the inverse problems for which the forward operator $F$ is not Fr\'echet differentiable. In this direction, Scherzer \cite{scherzer1995convergence} shown   that one can replace the Fr\'echet derivative $F'(u)$ in  \ref{LB}  by another linear operator $G_u$ that is sufficiently close to $F'(u)$ in some sense (see also \cite{kugler2003derivative, kugler2005derivative}). Later,  Clason et al.  \cite{clason2019bouligand} shown  that the linear operator $G_u$ can be taken from the Bouligand subdifferential of $F$ (see Definition 2.1).  Recently,     the extensions of  two-point gradient method (a fast Landweber type method) and steepest descent method   for non-smooth problems were proposed, respectively, in \cite{fu2023two, mittal2023modified}.

It can be observed that for the  conventional iterative regularization methods for solving inverse problems (such as the  Landweber method \cite{kaltenbacher2008iterative}, Levenberg-Marquardt method \cite{clason2019bouligand2}, conjugate gradient method \cite{kaltenbacher2008iterative}, non-stationary Tikhonov iterative regularization method \cite{Gaurav6} and iteratively regularized Gauss-Newton method \cite{kaltenbacher2008iterative}, among others),  one of the common challenges  is the significant computational burden at each iteration, mainly due to the need to process all the data, which can be excessively large and resource-intensive. To address this issue, a potential and effective approach  is the utilization of stochastic gradient descent (SGD), a method introduced by Robbins et al.  \cite{robbins1951stochastic}. Further, 
Jin et al. \cite{jin2018regularizing}  studied the convergence analysis  of  SGD method and demonstrated its regularizing nature  when there is a noise in the data. The basic version of SGD can be written as (see  \cite{jin2020convergence}) \begin{equation}\label{SGD}
    u_{k+1}^{\delta} = u_{k}^{\delta} - \mu_{k}F_{i_k}'(u_k^{\delta})^*(F_{i_k}(u_k^{\delta}) - y_{i_k}^{\delta}), \hspace{10mm} k \geq 0,
 \end{equation}
where $\mu_{k}$ is the corresponding step size and $i_{k}$ is uniformly drawn index from the index set $\{0,1,..., P-1\}.$ It can be noted that the method \ref{SGD} is a randomized version of \ref{LB}. For the method \ref{SGD}, the  stopping index $k(\delta)$ is defined using an a-priori stopping rule, i.e., the stopping index $k(\delta) \in \mathbb{N}$ satisfies
\begin{equation}\label{aprioriforSGD}
\lim_{\delta \rightarrow 0^{+}} k(\delta) = \infty \hspace{5mm} \text{and} \hspace{5mm} \lim_{\delta \rightarrow 0^{+}}\delta^2 \sum_{i=1}^{k(\delta)} \mu_{i} = 0.
\end{equation}
 We also refer to \cite{jahn2020discrepancy, jin2020convergence,lu2022stochastic} for some recent literature on SGD method  for linear as well nonlinear inverse problems. 
Motivated by \cite{jin2020convergence}, our main aim in this paper is to introduce a stochastic gradient descent (SGD) for non-smooth system of equation with inclusion of prior information in Hilbert space. More specifically, we propose a new method known as  \textit{stochastic data-driven Bouligand Landweber iteration} (SDBLI), which can be  formulated as 
 \begin{equation}\label{stmain1}
     u_{k+1}^{\delta} = u_{k}^{\delta} - \omega_{k}G_{i_k}(u_{k}^{\delta})^*(F_{i_k}(u_k^{\delta}) - y_{i_k}^{\delta}) - \lambda_{k}M'_{i_k}(u_{k}^{\delta})^*(M_{i_k}(u_k^{\delta}) - y_{i_k}^{\delta}), \hspace{5mm} k \geq 0,
     \end{equation}
     where $\omega_{k}$ is the step size, $\lambda_{k}$ is the weighted parameter, $G_{i}(u)$ is a Bouligand subdifferential of $F_{i}(u)$,  $i_k$ is an index  uniformly drawn from the index set $\{0,1,\cdot \cdot \cdot P-1\}$ and $M_{i}$  is a bounded linear operator   for an arbitrary $i$ belongs to the  index set $\{0,1,\cdot \cdot \cdot P-1\}$. 
By considering the case that the information of $\delta_0, \delta_1, \ldots, \delta_{P-1}$ is available, we choose the step size $\omega_k$ according to the rule
     \begin{equation}\label{new}
         \omega_k = \begin{cases}
             \text{constant}\neq 0 \in [\omega, \Omega] & \text{if} \hspace{2mm} \|F_{i_k}(u_k^{\delta}) - y_{i_k}^{\delta}\| > \tau \delta_{i_k}, \\
             0 & \text{otherwise,}
         \end{cases}
     \end{equation}
     where $\tau \geq 1.$     
      To include the  information of the forward operator $F_i$, we expect to use the training pairs 
     \[\{u^{(l)}, y_{i}^{(l)}\}_{l=1}^{N}, \hspace{5mm} \text{such that} \hspace{5mm} F_i(u^{(l)}) = y_{i}^{(l)}.\] 
      To be more specific, we establish the operator $M_i$ in such a way that it fulfills the condition $M_i(u^{(l)}) = y_{i}^{(l)}.$  The main difficulty here is that the mapping $u \mapsto G_{i}(u)$ is not continuous, which is an important aspect of the convergence analysis   presented in \cite{aspri2020data, jin2020convergence}. In order to eliminate this difficulty, motivated by the work of Clason et al. \cite{clason2019bouligand},  we perform a new convergence analysis of the method \ref{stmain1}  based on the concept of asymptotic stability  (see Definition \ref{def-AS}).

The main contributions of this paper are as follows:\newline
1.   Rather than concentrating on the analysis of a single equation, our attention is directed towards exploring stochastic methods for resolving systems that consist of large number of equations (i.e., $P$ is large in \ref{ststart}). Stochastic methods involve the random consideration of each equation within equation \ref{ststart} independently, which results in a more manageable memory requirement. This approach finds relevance in numerous real-world scenarios, particularly in applications like certain tomography techniques that involve multiple measurements (see \cite{natterer2001mathematics, olafsson2006radon, hanafy1991quantitative}). Furthermore, for large scale problems, the computational and storage overhead associated with constructing $M_i$ using a subset of the data is more efficient than that of $M$ constructed using all the available data.

2.  The method is applicable on non-smooth ill-posed inverse problems and this is the first attempt of construction of a stochastic method for non-smooth inverse problems. 

3.  We create a set of operators, denoted as $M_i,$ based on the  training pairs $(u^{(l)}, y^{(l)})$, where $y^{(l)} = (y_{0}^{(l)}, y_{1}^{(l)}, \cdots , y_{p-1}^{(l)})$, for which ${F}(u^{(l)}) = y^{(l)}.$ These operators can be viewed as the approximations of the forward operators $F_i$. The incorporation of data-driven terms constructed from $M_i$ allows us to include prior data information, which substantially enhances the accuracy of the inversion results, especially when the data set closely approximates the exact solution.

The rest of the paper is structured as follows. In Section \ref{sec-2}, we list   primary  definitions and notations needed in our work. Section 3 is divided into three subsections. In Section \ref{assumptions}, we frame the assumptions required for our work and discuss a proposition devoted to  the monotonicity of the error term for the noisy case. In Section \ref{secfornonnoisy}, we discuss the  convergence of the iterates for our method  when $\delta =0.$ The convergence analysis for the noisy data is given in Section  \ref{noisy case}. In Section \ref{NE}, we  discuss an example that fulfills the assumptions of our work.  The last section is devoted to some concluding remarks. 

\section{Preliminaries}\label{sec-2}
In this section, we recall some basic definitions and notations relevant to our work, see  \cite{clason2019bouligand, jin2020convergence, outrata2013nonsmooth}  for more details.\\
We use the notation $u_k$ to represent the iterates of our method for the exact data $y^{\dagger}.$ Additionally, we refer to the filtration generated by the random indices $\{i_1, ..., i_{k - 1}\}$ up to the $(k - 1)$th iteration as $\mathcal{F}_k$. Because we randomly select the indices $i_k$, the iteration \(u_k^\delta\) in the SDBLI method becomes a random process. To assess its convergence,   multiple methods are available. In our case, we will use the mean squared norm defined by \(\mathbb{E}[\|\cdot\|^2]\), where \(\mathbb{E}[\cdot]\) represents the expectation taken with respect to the filtration \(\mathcal{F}_k\). It is important to note that the iterates $u_k^{\delta}$ is measurable with respect to $\mathcal{F}_k.$\newline We will regularly employ the following well-known identities: for any random variables $\Theta, \Theta_1, \Theta_2$ and $\Phi$, and real constants $c_1, c_2$, we have
\begin{itemize}
    \item [(a)] $ \mathbb{E}[\Theta] = \mathbb{E}[\mathbb{E}[\Theta| \mathcal{F}_k]].$
    \item [(b)]  Additive:  $ \mathbb{E}[\Theta_1 + \Theta_2 | \mathcal{F}_k] = \mathbb{E}[\Theta_1| \mathcal{F}_k] + \mathbb{E}[\Theta_2| \mathcal{F}_k].$ \\
     Linearity: $ \mathbb{E}[c_1\Theta + c_2 | \mathcal{F}_k] = c_1\mathbb{E}[\Theta| \mathcal{F}_k] + c_2.$
    \item [(c)] If $\Theta \leq \Phi$, then $\mathbb{E}[\Theta | \mathcal{F}_k] \leq \mathbb{E}[\Phi | \mathcal{F}_k]$ almost surely.
\end{itemize}
Here the expression $E[\Theta| \mathcal{F}_k]$ represents the expected value of $\Theta$ given the information contained in $\mathcal{F}_k$. \\
For some $u \in U$ and $\sigma > 0,$ we denote  the closed and open balls in $U$ of radius $\sigma$ centered at $u$, respectively by $\overline{\mathcal{B}}_{U}(u, \sigma)$ and $\mathcal{B}_{U}(u, \sigma).$ Let  $\mathcal{D}(u^{\dagger}, \sigma)$ denote  the set of all the solutions of  \ref{stcombined} in  $\overline{\mathcal{B}}
_U(u^{\dagger}, \sigma)$, i.e., 
\begin{equation*}
    \mathcal{D}(u^{\dagger}, \sigma):= \{u \in \overline{\mathcal{B}}_U(u^{\dagger}, \sigma) : F(u) = y^{\dagger}\}.
\end{equation*}
It is trivial to note that  $u^{\dagger} \in \mathcal{D}(u^{\dagger}, \sigma) $ for all $\sigma > 0.$  
Next, we recall the concepts of Bouligand subdifferential and asymptotic stability. Let $\mathbb{B}(U, Y)$ denote  the space of all bounded linear operators from $U$ to $Y$.
 
\begin{defi}[Bouligand subdifferential]\label{BG def}
Let 
\begin{center}
    $W:=\{v \in U: F_{i}:U \rightarrow Y\ \text{is G\^ateaux differentiable at} \hspace{1mm} v\}.$ 
\end{center} 
Then the Bouligand subdifferential of $F_i$ at $u$ is defined as 
\begin{center}
    $\partial_{B}F_{i}(u) = \{G_{i}(u) \in \mathbb{B}(U, Y): \ \text{there exists}\  \{u_k\} \subset  W \hspace{1mm} \text{such that}$ \newline $\hspace{20mm}\ \ \ \  \ \ \ u_k \rightarrow u \ \  \text{and} \ \  F_{i}'(u_k;h) \rightarrow G_{i}(u)h \in Y \  \text{for all}\   h \in U \}.\hspace{-30mm}$
\end{center}
\end{defi}
We end this section by defining the notion of asymptotic stability. This definition is inspired  from the work of  \cite{clason2019bouligand}.
\begin{defi}[Asymptotic stability]\label{def-AS} For some $\delta>0$, let $\{u_{k}^{\delta}\}_{k \leq k(\delta)}$ be a sequence (finite or infinite), where $k(\delta)$ is given by $\ref{aprioriforSGD}$, generated through an iterative method $I$. This means $\ddot{K}:=\lim_{j\to\infty}k(\delta_{n_j})=\infty$ for any subsequence $\{\delta_{n_j}\}_{j \in \mathbb{N}}$ of any positive zero sequence $\{\delta_n\}_{n \in \mathbb{N}}$.
 Then the method $I$ is said to be  \textit{asymptotically stable}  if  the following assertions are satisfied:
    \begin{itemize}
        \item [(i)] For all $0\leq k<\ddot{K}$, we have \begin{equation}\label{(i)}
            u_{k}^{\delta_{{n}_j}} \rightarrow \hat{u}_k \in U \hspace{1mm} \text{as} \hspace{1mm} j  \rightarrow \infty \hspace{2mm} 
        \end{equation}
        for some $\hat{u}_k$ almost surely belongs to $\overline{\mathcal{B}}_U(u^{\dagger}, \sigma)$.
        \item [(ii)]   there exist a $\hat{u} \in \mathcal{D}(u^{\dagger}, \sigma)$ such that 
        \begin{equation}
            \hat{u}_k \rightarrow \hat{u} \in U \hspace{2mm} \text{as} \hspace{2mm} k \rightarrow \infty            
        \end{equation}
        almost surely.
    \end{itemize}
    
\end{defi}


\section{Stochastic data-driven Bouligand Landweber iteration}
In this section, we discuss the convergence analysis of   stochastic data-driven Bouligand Landweber iteration (SDBLI) method  under certain conditions. We  recall that the SDBLI method  presented in introduction can be written as  
\begin{equation}\label{stmain}
     u_{k+1}^{\delta} = u_{k}^{\delta} - \omega_{k}G_{i_k}(u_{k}^{\delta})^*(F_{i_k}(u_k^{\delta}) - y_{i_k}^{\delta}) - \lambda_{k}M'_{i_k}(u_{k}^{\delta})^*(M_{i_k}(u_k^{\delta}) - y_{i_k}^{\delta}), \hspace{5mm} k \geq 0,
     \end{equation}
where $G_{i_k}$ is a Bouligand subdifferential of $F_{i_k}$ and the index $i_k$ is drawn randomly from the index set $\{0, 1, ..., P-1\}$. 
 We need the following assumptions concerning the operators $F_{i}$ and $ M_{i}.$
\subsection{Assumptions}\label{assumptions}
We assume that for $i \in \{0,1,2 \cdot \cdot \cdot P-1\}$, there holds:\newline
(A1) $F_{i}: U \rightarrow Y$ is completely continuous. \newline
    (A2)  for any $u \in \overline{\mathcal{B}}_U(u^{\dagger}, \sigma)$,  $F_{i}(u)$ has a  Bouligand subdifferential $G_{i}(u)$
    such that
    \begin{equation}\label{lipbdonG}
        \|G_{i}(u)\| \leq L_F,
    \end{equation}
    where $L_F>0$ is a constant. Further, for any $u \in \overline{\mathcal{B}}_U(u^{\dagger}, \sigma)$,  $M_{i}(u)$ has continuous Fr\'echet derivative $H_i(u):=M_{i}'(u)$ such that
    \begin{equation}\label{lipbdonH}
        \|M_{i}'(u)\| \leq L_M,
    \end{equation}
    where $L_M>0$ is a constant.\newline
(A3) $M_{i}$ cannot fully explain $F_{i}$ for the true data, i.e., there exists a constant  $C_N>0$ such that 
    \begin{equation}
        \|M_{i}(u^{\dagger}) - y^{\dagger}_{i}\| \geq C_N.
    \end{equation}
(A4) there exists a positive constant $\mu>0$ such that  \begin{equation}\label{tangential cone}
\|F_{i}(u) - F_{i}(\Tilde{u}) - G_{i}(u)(u - \Tilde{u})\| \leq \mu \|F_{i}(u) - F_{i}(\Tilde{u})\| \hspace{5mm} \forall  u, \Tilde{u} \in \overline{\mathcal{B}}_U(u^{\dagger}, \sigma).
\end{equation}  
    (A5) there exist Banach spaces $Z_1, Z_2$ such that 
    \begin{equation*}
        \cup_{u\in U} \mathcal{R}(G_i(u)^{*}) \subset Z_1,\ \ \  \text{and}\ \ \       
         \cup_{u\in U} \mathcal{R}(H_i(u)^{*}) \subset Z_2,
    \end{equation*}
    with $Z_1$ and  $Z_2$ compactly contained in $U$. Moreover, there exist constants $\hat{L}_F >0$ and $\hat{L}_M> 0$ such that 
    \begin{center}
        $\|G_{i}(u)^*\|_{\mathbb{B}(Y, Z_1)} \leq \hat{L}_F$, \ \ \text{and}\ \        $\|H_{i}(u)^*\|_{\mathbb{B}(Y, Z_2)} \leq \hat{L}_M$,
    \end{center}
    for all $u \in \overline{\mathcal{B}}_U(u^{\dagger}, \sigma).$ 
\begin{rema}
    Although, our analysis is currently based on the  Fr\'echet differentiability of $M_{i},$ it is important to note that our approach can also be extended to the case when the operator $M_{i}$ is not necessarily smooth or not   G\^ateaux differentiable.
\end{rema}
\begin{rema}
   Given that $M_{i}$ is a bounded linear operator, the requirements specified in assumption (A5) related to $M_{i}$ can be substituted with the condition that $M_i$ is compact.
\end{rema}
\begin{rema} We observe that, for each $F_i$,  we are neither assuming the smoothness of $F_i$ nor the continuity of the mappings  $u \mapsto G_{i}(u)$ in our analysis. 
\end{rema} 
\begin{lem}
   Let assumptions $(A1)$-$(A3)$ be satisfied. Then there exist a constant $C_{M}^{\delta}>0$ depending on $\delta$ such that
    \begin{equation}\label{mleqcm}
        \|M_{i}(u_{k}^{\delta}) - y_{i}^{\delta}\| \leq C_{M}^{\delta} \hspace{5mm} \forall  u_{k}^{\delta} \in \overline{\mathcal{B}}_U(u^{\dagger}, \sigma),
    \end{equation}
    where $i = 0,1,2, \cdots, P-1$ and $u_k^{\delta}$ is given by $\ref{stmain}$. 
\end{lem}
\begin{proof}
    See \cite[Lemma 2.2]{aspri2020data} for the proof.
\end{proof}
To prove the convergence analysis of SDBLI method, firstly, we derive an important inequality in the following proposition.
\begin{prop}\label{decprop}
  Let assumptions $(A1)$-$(A4)$ be satisfied, $u_0 \in \overline{\mathcal{B}}_U(u^{\dagger}, \sigma)$ and let $\Omega \geq \omega$ be  positive constants.
  Further, let there exists a constant  $C_{\lambda}^{\delta}>0$
such that
  \begin{equation}\label{lamdd}
      \lambda_k \leq C_{\lambda}^{\delta}\|F_{i}(u_k^{\delta}) - y_{i}^{\delta}\|^2, \hspace{5mm} \forall  i \in \{0, 1, 2, \cdots, P-1\},
  \end{equation}
  where   $\lambda_k$ is same as in $\ref{stmain}$. Additionally, for any $\delta>0$, let $C_M^{\delta}$ in Lemma $1$ and other constants be such that 
  \begin{equation}\label{lambdasigma}
      \lambda_k L_M C_M^{\delta} \leq \sigma,
  \end{equation} 
    \begin{equation}\label{3.8(1)}
  \omega(1-L_{F}^2\Omega -\mu)\geq 2\sigma L_{M}C_{M}^{\delta}C_{\lambda}^{\delta}+(1+\mu)\frac{\Omega}{\tau},
  \end{equation}
  where $\tau$ is same as in $\ref{new}$. Let   $\hat{u}$ be a solution  of  $\ref{stcombined}$. 
Then for any $\delta >0$ and any step size $\omega_k \in [\omega, \Omega],$   there holds
  \begin{equation}\label{3.9}
      \begin{split}
         \mathbb{E}[\|u_{k+1}^{\delta} - \hat{u}\|^2] -  \mathbb{E}[\|u_{k}^{\delta} - \hat{u} \|^2] \leq    \hspace{70mm} \\ -2\bigg[\omega(1-L_{F}^2\Omega -\mu)  -2\sigma L_{M}C_{M}^{\delta}C_{\lambda}^{\delta}-(1+\mu)\frac{\Omega}{\tau}\bigg]\mathbb{E}[\|F(u_k^{\delta}) - y^{\delta}\|^{2}],
    \end{split}
  \end{equation}
   and $u_{k+1}^{\delta}\in \overline{\mathcal{B}}_U(u^{\dagger}, \sigma)$ almost surely.

\end{prop}
\begin{proof}
We prove the result by using the induction principle. By assumption $u_{0}^{\delta}= u_{0} \in \overline{\mathcal{B}}_{U}(u^{\dagger}, \sigma)$. Let $ u_{k}^{\delta} \in \overline{\mathcal{B}}_{U}(u^{\dagger}, \sigma)$. Then by 
   incorporating the definition of 	$u_{k}^{\delta}$ in \ref{stmain}, we get 
    $$
   \|u_{k+1}^{\delta} - \hat{u}\|^2 -  \|u_{k}^{\delta} - \hat{u} \|^2  = 2(u_k^{\delta} - \hat{u}, u_{k+1}^{\delta} - u_{k}^{\delta}) + \|u_{k+1}^{\delta} - u_{k}^{\delta}\|^2\hspace{40mm} $$
   $$= 2(u_k^{\delta} - \hat{u}, -\omega_{k}G_{i_k}(u_k^{\delta})^*(F_{i_k}(u_k^{\delta}) - y_{i_k}^{\delta}) - \lambda_{k}M_{i_k}'(u_k^{\delta})^*(M_{i_k}(u_k^{\delta}) - y_{i_k}^{\delta}))\hspace{20mm} $$ $$+ \| \omega_{k}G_{i_k}(u_k^{\delta})^*(F_{i_k}(u_k^{\delta}) - y_{i_k}^{\delta}) +\lambda_{k}M_{i_k}'(u_k^{\delta})^*(M_{i_k}(u_k^{\delta}) - y_{i_k}^{\delta})\|^2 $$ $$\leq -2\omega_k(G_{i_k}(u_{k}^{\delta})(u_k^{\delta} - \hat{u}), F_{i_k}(u_k^{\delta}) - y_{i_k}^{\delta}) -2\lambda_k(M_{i_k}'(u_{k}^{\delta})(u_k^{\delta} - \hat{u}), M_{i_k}(u_k^{\delta}) - y_{i_k}^{\delta})\hspace{2mm}  $$ $$
    + 2(\| \omega_{k}G_{i_k}(u_k^{\delta})^*(F_{i_k}(u_k^{\delta}) - y_{i_k}^{\delta})\|^2 + \| \lambda_{k}M_{i_k}'(u_k^{\delta})^*(M_{i_k}(u_k^{\delta}) - y_{i_k}^{\delta})\|^2)$$ \begin{equation}\label{3.10}
         = 2(A_F + A_M),\hspace{105mm}
         \end{equation}
         where we have $$ A_F = -\omega_k(G_{i_k}(u_{k}^{\delta})(u_k^{\delta} - \hat{u}), F_{i_k}(u_k^{\delta}) - y_{i_k}^{\delta}) + \| \omega_{k}G_{i_k}(u_k^{\delta})^*(F_{i_k}(u_k^{\delta}) - y_{i_k}^{\delta})\|^2,$$ and 
         $$ A_M = -\lambda_k(M_{i_k}'(u_{k}^{\delta})(u_k^{\delta} - \hat{u}), 
       M_{i_k}(u_k^{\delta}) - y_{i_k}^{\delta}) + \| \lambda_{k}M_{i_k}'(u_k^{\delta})^*(M_{i_k}(u_k^{\delta}) -y_{i_k}^{\delta})\|^2 .$$ 
        Next,  we individually estimate the quantities  $A_F$ and $A_M$.   More precisely, we derive a bound for each of them in relation to the square norm of the residual of $F$ at the $k$-th iteration.
       By utilizing \ref{lipbdonG} and \ref{tangential cone}, we note that
      $$
  A_F  =  \omega_k(F_{i_k}(u_k^{\delta}) - y_{i_k}^{\delta}-G_{i_k}(u_{k}^{\delta})(u_k^{\delta} - \hat{u}), F_{i_k}(u_k^{\delta}) - y_{i_k}^{\delta}) \hspace{55mm} $$ $$-\omega_k(F_{i_k}(u_k^{\delta}) - y_{i_k}^{\delta},  F_{i_k}(u_k^{\delta}) - y_{i_k}^{\delta} )+ \| \omega_{k}G_{i_k}(u_k^{\delta})^*(F_{i_k}(u_k^{\delta}) - y_{i_k}^{\delta})\|^2 $$
$$=\omega_k(F_{i_k}(u_k^{\delta}) - y_{i_k}^{\dagger}-G_{i_k}(u_{k}^{\delta})(u_k^{\delta} - \hat{u}),  F_{i_k}(u_k^{\delta}) - y_{i_k}^{\delta})   -\omega_k(y_{i_k}^{\delta}-y_{i_k}^{\dagger}, F_{i_k}(u_k^{\delta}) - y_{i_k}^{\delta})\hspace{-5mm}$$ $$-\omega_k(F_{i_k}(u_k^{\delta}) - y_{i_k}^{\delta},  F_{i_k}(u_k^{\delta}) - y_{i_k}^{\delta} )+ \| \omega_{k}G_{i_k}(u_k^{\delta})^*(F_{i_k}(u_k^{\delta}) - y_{i_k}^{\delta})\|^2 $$ $$\leq-\omega_k\|F_{i_k}(u_k^{\delta}) - y_{i_k}^{\delta}\|\big[(1-L_{F}^2\omega_k)\|F_{i_k}(u_k^{\delta}) - y_{i_k}^{\delta}\| -\mu\|F_{i_k}(u_k^{\delta}) - y_{i_k}^{\dagger}\|-\delta_{i_k}\big]\hspace{5mm}$$ \begin{equation}\label{3.11}
 \leq-\omega_k\|F_{i_k}(u_k^{\delta}) - y_{i_k}^{\delta}\|[(1-L_{F}^2\omega_k -\mu)\|F_{i_k}(u_k^{\delta}) - y_{i_k}^{\delta}\|-(1+\mu)\delta_{i_k}].
 \end{equation}
Next, in order to approximate the term $A_M$, we utilize  \ref{lipbdonH} and \ref{mleqcm}-\ref{lambdasigma} to attain 
   $$
        A_M  = -\lambda_k(M_{i_k}'(u_{k}^{\delta})(u_k^{\delta} - \hat{u}), 
        M_{i_k}(u_k^{\delta}) - y_{i_k}^{\delta}) + \| \lambda_{k}M_{i_k}'(u_k^{\delta})^*(M_{i_k}(u_k^{\delta}) -y_{i_k}^{\delta})\|^2 $$ 
      $$\leq \lambda_{k}L_M \|M_{i_k}(u_k^{\delta}) - y_{i_k}^{\delta}\|(\sigma +\lambda_{k}L_M \|M_{i_k}(u_k^{\delta}) - y_{i_k}^{\delta}\| )\hspace{26mm} $$ 
      $$\leq   \lambda_{k}L_M \|M_{i_k}(u_k^{\delta}) - y_{i_k}^{\delta}\|(\sigma +\lambda_{k}L_M  C_M^{\delta} ) \hspace{48mm}
       $$\begin{equation}\label{3.12} \leq  2\sigma\lambda_{k}L_{M}C_{M}^{\delta} \leq 2\sigma L_{M}C_{M}^{\delta}C_{\lambda}^{\delta} \|F_{i_k}(u_k^{\delta}) - y_{i_k}^{\delta}\|^{2}. \hspace{36mm}
       \end{equation}
   We plug the estimates \ref{3.11} and \ref{3.12} in \ref{3.10} and use the fact  that $\omega_k\in  [\omega, \Omega]$ along with \ref{new} to get 
   $$
            \|u_{k+1}^{\delta} - \hat{u}\|^2 -  \|u_{k}^{\delta} - \hat{u} \|^2 \leq -2(\omega(1-L_{F}^2\Omega -\mu)-2\sigma L_{M}C_{M}^{\delta}C_{\lambda}^{\delta}  )\|F_{i_k}(u_k^{\delta}) - y_{i_k}^{\delta}\|^{2}  $$  \begin{equation*}\label{C_F}+\frac{2\Omega(1+\mu)}{\tau}\|F_{i_k}(u_k^{\delta}) - y_{i_k}^{\delta}\|^2. 
      \end{equation*}
 Furthermore, due to the measurability of $x_k$ concerning $\mathcal{F}_k$ and the definition of $F$, we can deduce that
 \begin{equation*}\label{stC_F}
  \begin{split}
       \mathbb{E}[\|u_{k+1}^{\delta} - \hat{u}\|^2 -  \|u_{k}^{\delta} - \hat{u} \|^2 | \mathcal{F}_k]\hspace{80mm} \\\leq -2\left[\omega(1-L_{F}^2\Omega -\mu)-2\sigma L_{M}C_{M}^{\delta}C_{\lambda}^{\delta} - (1+\mu)\frac{\Omega}{\tau}  \right]\|F(u_k^{\delta}) - y^{\delta}\|^{2}.
  \end{split}  
   \end{equation*}
Finally, by considering the full conditional, we reach at
    \begin{equation}\label{finalstC_F}
    \begin{split}
         \mathbb{E}[\|u_{k+1}^{\delta} - \hat{u}\|^2] -  \mathbb{E}[\|u_{k}^{\delta} - \hat{u} \|^2]\hspace{80mm}\\ \leq -2\left[\omega(1-L_{F}^2\Omega -\mu)-2\sigma L_{M}C_{M}^{\delta}C_{\lambda}^{\delta}  -  (1+\mu)\frac{\Omega}{\tau} \right]\mathbb{E}[\|F(u_k^{\delta}) - y^{\delta}\|^{2}].
    \end{split}    
   \end{equation}
 This gives \ref{3.9}. Furthermore,  \ref{3.8(1)} confirms that
\begin{equation}\label{Edec}
    \mathbb{E}[\|u_{k+1}^{\delta} - \hat{u}\|^2] \leq \mathbb{E}[\|u_{k}^{\delta} - \hat{u} \|^2].
\end{equation}
 By repetitive application of \ref{Edec} to the case $\hat{u} = u^{\dagger},$ we obtain
\[\mathbb{E}[\|u_{k+1}^{\delta} - u^{\dagger}\|^2] \leq \mathbb{E}[\|u_{k}^{\delta} - u^{\dagger} \|^2]\leq \ldots \leq \mathbb{E}[\|u_{0}^{\delta} - u^{\dagger} \|^2] = \|u_{0}^{\delta} - u^{\dagger} \|^2 \leq \sigma^2.\]
This implies that $u_{k+1}^{\delta} \in \overline{\mathcal{B}}_U(u^{\dagger}, \sigma)$ almost surely. This completes the proof. 
 \end{proof}
\subsection{Convergence for exact data}\label{secfornonnoisy} In case of exact data, we have the next result, which is a direct consequence of Proposition \ref{decprop}.
\begin{cor}\label{decandsum}
    Let assumptions of Proposition \ref{decprop} be satisfied. Additionally,  assume that 
    \begin{equation}\label{stdeccondtition}
        \omega(1-L_{F}^2\Omega -\mu) > 2\sigma L_{M}C_{M}^{\delta}C_{\lambda}^{\delta}.
    \end{equation}
    Then the iterates $\{u_k\}_{k \geq 1}$ in $\ref{stmain}$ defined for the exact data $y^{\dagger}$ satisfies
     \begin{equation}\label{decfordeltazero}
     \begin{split}
         \mathbb{E}[\|u_{k+1} - \hat{u}\|^2] \leq \mathbb{E}[\|u_{k} - \hat{u} \|^2], 
         \end{split}
   \end{equation}
   \begin{equation}\label{stsumforzero}
   \sum_{k=1}^{\infty} \mathbb{E}[\|F(u_k) - y^{\dagger} \|^2] < \infty.      
   \end{equation}
\end{cor}
\begin{proof}
    For $\delta = 0$ in \ref{finalstC_F}, we get
      \begin{equation}\label{finalstC_Fzero}
    \begin{split}
         \mathbb{E}[\|u_{k+1} - \hat{u}\|^2] -  \mathbb{E}[\|u_{k} - \hat{u} \|^2] \leq -2(\omega(1-L_{F}^2\Omega -\mu)-2\sigma L_{M}C_{M}^{\delta}C_{\lambda}^{\delta}  )\mathbb{E}[\|F(u_k) - y^{\dagger}\|^{2}].
    \end{split}    
   \end{equation}
   By pulgging  \ref{stdeccondtition} in the last inequality, we get \ref{decfordeltazero}. For deriving \ref{stsumforzero},
  we take the summation of  \ref{finalstC_Fzero} from $k = 1$ to $\infty$  to attain that 
   \begin{equation}\label{stsumdeltazero}
       \sum_{k=1}^{\infty}\mathbb{E}[\|F(u_k) - y^{\dagger}\|^2] \leq \frac{1}{\Tilde{C}_F} \|\hat{u}-u_0\|^2,
   \end{equation}
   where $\Tilde{C}_F =2(\omega(1-L_{F}^2\Omega -\mu)-2\sigma L_{M}C_{M}^{\delta}C_{\lambda}^{\delta}). $ This is the desired inequality \ref{stsumforzero}.
\end{proof}
In the following proposition, we show that the sequence $\{u_k\}_{k \geq 1}$ generated by $\ref{stmain}$ is a Cauchy sequence. Its proof closely follows the one presented in \cite[Lemma 3.3]{jin2020convergence}, with some necessary modifications to accommodate the data-driven term.
\begin{lem}\label{ukcauchy}
    Under the assumptions of Proposition $\ref{decprop}$, for the exact data $y^{\dagger},$ the sequence $\{u_k\}_{k \geq 1}$ generated by $\ref{stmain}$ is almost surely a Cauchy sequence. 
\end{lem}
\begin{proof}
 
 For a solution $\hat{u}$ to \ref{stcombined}, we define $e_k := u_k - u^{\dagger}$. As indicated in \ref{decfordeltazero}, $\mathbb{E}[\| e_k\|^2]$ is a monotonically decreasing sequence. Being positive, let  it converge to some $\epsilon \geq 0$. Our main target is to demonstrate that the sequence $\{u_k\}_{k\geq 1}$ is indeed a Cauchy sequence. To begin, it is important to note that $\mathbb{E}[(\cdot , \cdot )]$ serves as an inner product. For any $j \geq k$, we choose an index $\ell$ with $j \geq \ell \geq k$ such that
\begin{equation}\label{choiceofl}
    \mathbb{E}[\| y^\dagger - F(u_\ell)\|^2] \leq \mathbb{E}[\| y^\dagger - F(u_i)\|^2] \quad \forall k \leq i \leq j.
\end{equation}
Using inequality $\mathbb{E}[\| e_j - e_k\|^2]^{1/2} \leq \mathbb{E}[\| e_j - e_\ell\|^2]^{1/2} + \mathbb{E}[\| e_\ell - e_k\|^2]|^{1/2}$ and the identities
\begin{align*}
    \mathbb{E}[\| e_j - e_\ell\|^{2}] &= 2\mathbb{E}[( e_\ell - e_j , e_\ell )] + \mathbb{E}[\| e_j\|^{2}] - \mathbb{E}[\| e_\ell\|^{2}], \\
    \mathbb{E}[\| e_\ell - e_k\|^{2}] &= 2\mathbb{E}[( e_\ell - e_k, e_\ell )] + \mathbb{E}[\| e_k\|^{2}] - \mathbb{E}[\| e_\ell\|^{2}],
\end{align*}
it is sufficient to establish that $\mathbb{E}[\| e_j - e_\ell\|^{2}]$ and $\mathbb{E}[\| e_\ell - e_k\|^{2}]$ tend to zero as $k \rightarrow \infty$.
As $k \rightarrow \infty$, the last two terms on the right-hand sides of the preceding two identities tend to $\epsilon - \epsilon = 0$. This is due to the monotone convergence of $\mathbb{E}[\| e_k\|^{2}]$ to $\epsilon$, as  discussed above. Next, we proceed to demonstrate that the term $\mathbb{E}[( e_\ell - e_k, e_\ell )]$ also approaches to zero as $k \rightarrow \infty$. In fact, by virtue of the definition of $u_k$, we have the following
\begin{align*}
    e_\ell - e_k &= \sum_{i=k}^{\ell-1} (e_{i+1} - e_i) \\
    &= \sum_{i=k}^{\ell-1} \omega_i G_{i_i}(u_i)^*(y_{i_i}^\dagger - F_{i_i}(u_i)) + \lambda_iM'_{i_i}(u_i)^*(y_{i_i}^\dagger - M_{i_i}(u_i)).
\end{align*}
By utilizing   the triangle inequality and the Cauchy-Schwarz inequality, we can write
$$
  \ |\mathbb{E}[( e_l - e_k, e_l )]|  \leq \sum_{i=k}^{\ell - 1} \left[ \omega_i |\mathbb{E}[( G_{i_i}(u_i)^*\!(y_{i_i}^\dagger - F_{i_i}(u_i)), e_l) ]| + \lambda_i|\mathbb{E}( M'_{i_i}(u_i)^*\!(y_{i_i}^\dagger - M_{i_i}(u_i)) , e_l )]| \right]   $$
 \begin{equation}\label{neww}
   = H_F + H_M.\hspace{80mm}
\end{equation}
To this end, we separately estimate  both $H_F$ and $H_M$. Firstly, we note that 
  \begin{align*}
    H_F &= \sum_{i=k}^{\ell - 1} \omega_i |\mathbb{E}[( y_{i_i}^\dagger - F_{i_i}(u_i), G_{i_i}(u_i)(u^{\dagger} - u_i + u_i - u_\ell))]| \\
    &= \sum_{i=k}^{\ell - 1} \omega_i |\mathbb{E}[( y^\dagger - F(u_i), G(u_i)(u^{\dagger} - u_i + u_i - u_\ell))]| \\
    &\leq \sum_{i=k}^{\ell - 1} \omega_i\mathbb{E}[\| y^\dagger - F(u_i)\|^2]^{1/2} \mathbb{E}[\| G(u_i)(u^{\dagger} - u_i)\|^2]^{1/2} \\
    &\quad + \sum_{i=k}^{\ell - 1} \omega_i\mathbb{E}[\| y^\dagger - F(u_i)\|^2]^{1/2} \mathbb{E}[\| G(u_i)(u_i - u_\ell)\|^2]^{1/2} = H_{F_1} + H_{F_2}.
\end{align*}
By using the  tangential cone condition \ref{tangential cone}, we bound the first term $H_{F_1}$  by
\begin{align*}
   H_{F_1}&\leq (1 + \mu ) \sum_{i=k}^{\ell - 1} \omega_i\mathbb{E}[\| y^\dagger - F(u_i)\|^2]^{1/2} \mathbb{E}[\| F(u^{\dagger}) - F(u_i)\|^2]^{1/2} \\
    &= (1 + \mu )\Omega \sum_{i=k}^{\ell - 1} \mathbb{E}[\| y^\dagger - F(u_i)\|^2].
\end{align*}
Likewise, we bound the term $H_{F_2}$ by the triangle inequality and the choice of \(\ell\) in \ref{choiceofl} as
\begin{align*}
    H_{F_2} &\leq (1 + \mu ) \sum_{i=k}^{\ell - 1} \omega_i\mathbb{E}[\| y^\dagger - F(u_i)\|^2]^{1/2} \mathbb{E}[\| (F(u_\ell) - y^\dagger) + (y^\dagger - F(u_i))\|^2]^{1/2} \\
    &\leq 2(1 + \mu)\Omega \sum_{i=k}^{\ell - 1} \mathbb{E}[\| y^\dagger - F(u_i)\|^2].
\end{align*}
By substituting the values of $H_{F_1}$ and $H_{F_2}$ in $H_F$, we obtain that 
\begin{equation}\label{h_f}
 H_F\leq 3(1 + \mu)\Omega \sum_{i=k}^{\ell - 1} \mathbb{E}[\| y^\dagger - F(u_i)\|^2].
\end{equation}
Similar to $H_F$, we estimate $H_M$  as
\begin{align*}
    H_M & = \sum_{i=k}^{\ell - 1} \lambda_i|\mathbb{E}( M'_{i_i}(u_i)^*(y_{i_i}^\dagger - M_{i_i}(u_i)) , e_l )]| \\
    &= \sum_{i=k}^{\ell - 1} \lambda_i|\mathbb{E}(\!(y_{i_i}^\dagger - M_{i_i}(u_i) , M'_{i_i}(u_i)(u^{\dagger}- u_l ))]| \\   
    & \leq  \sum_{i=k}^{\ell - 1} \lambda_i\mathbb{E}[\|(y_{i_i}^\dagger - M_{i_i}(u_i)\|^2]^{\frac{1}{2}} \mathbb{E}[\| M'_{i_i}(u_i)(u^{\dagger}- u_l )\|^2]^{\frac{1}{2}}\\  
     & \leq  \sum_{i=k}^{\ell - 1} \lambda_i\mathbb{E}[\|(y^\dagger - M(u_i)\|^2]^{\frac{1}{2}} \mathbb{E}[\| M'(u_i)(u^{\dagger}- u_l )\|^2]^{\frac{1}{2}}. \end{align*}
This along with \ref{mleqcm}, \ref{lipbdonH}     and the result that  $u_0\in \overline{\mathcal{B}}_U(u^{\dagger}, \sigma)$ along with \ref{decfordeltazero} and \ref{lamdd} provide that 
      $$  H_M   \leq \sigma L_M C_{M}^0   \sum_{i=k}^{\ell - 1} \lambda_i = \sigma L_M C_{M}^0 C_{\lambda}^0    \sum_{i=k}^{\ell - 1}\mathbb{E}[\| y^\dagger - F(u_i)\|^2].$$
Plugging the last estimate and \ref{h_f} in \ref{neww} to attain that 
$$|\mathbb{E}[( e_\ell - e_k, e_\ell )]| \leq (3(1 + \mu )\Omega + \sigma L_M C_{M}^0 C_{\lambda}^0)\sum_{i=k}^{\ell - 1} \mathbb{E}[\| y^\dagger - F(u_i)\|^2].$$
Similarly, one can deduce that 
$$ |\mathbb{E}[( e_j - e_\ell, e_\ell )]| \leq (3(1 + \mu )\Omega + \sigma L_M C_{M}^0 C_{\lambda}^0) \sum_{i=l}^{j - 1} \mathbb{E}[\| y^\dagger - F(u_i)\|^2].$$
These two evaluations along with Corollary \ref{decandsum} lead us to conclude that $\mathbb{E}[\| e_j - e_\ell\|^{2}]$ and $ \mathbb{E}[\| e_\ell - e_k\|^{2}]$ tend to zero as \(k \rightarrow \infty\). Consequently, both the sequences \(\{ e_k\}_{k\geq 1}\) and \(\{ u_k\}_{k\geq 1}\) are almost surely Cauchy sequences. This completes the proof.
\end{proof}
In the following theorem, we show the convergence of   non-noisy iterates.
\begin{thm}
     Under the assumptions of Proposition $\ref{decprop}$, the SDBLI method  $\ref{stmain}$ corresponding to $\delta =0$ generates a sequence of iterates that converges almost surely to an element of $\mathcal{D}(u^{\dagger}, \sigma)$ in $U,$ i.e.,
\[
\lim_{k\rightarrow \infty} \mathbb{E} [\| u_k - \hat{u} \|^{2}] = 0, \hspace{5mm} \hat{u} \in \mathcal{D}(u^{\dagger}, \sigma )\ \text{almost surely}.
\]

\end{thm}
\begin{proof}

 From Lemma \ref{ukcauchy}, it is clear that the generated sequence   $\{u_k\}_{k\in \mathbb{N}}$     is a Cauchy sequence. Therefore,  it has a limit and we denote it by $\overline{u}$ that belongs to $U$ almost surely. As $F$ is completely continuous, we must have 
\[{F}(u_k) \rightarrow {F}(\overline{u}) \hspace{5mm} \text{as} \hspace{5mm} k \rightarrow \infty.\]
In addition, we know from \ref{stsumforzero} that 
\[
\lim_{k\rightarrow \infty} \mathbb{E} [\|  {F}(u_k) - y^{\dagger} \|^{2}] = 0,
\]
and hence $y^{\dagger} =  {F}(\overline{u}).$ Using Proposition \ref{decprop}, we know that  $u_k \in \overline{\mathcal{B}}_U(u^\dagger, \sigma)$ for all $k \geq 0$ almost surely, therefore,  it can be concluded that $\overline{u}$ also belongs to $\overline{\mathcal{B}}_U(u^\dagger, \sigma)$ almost surely. Consequently, $\overline{u}$ is within the set $\mathcal{D}(u^\dagger, \sigma)$ almost surely, and this finishes the proof.
\end{proof}


\subsection{Regularization property}\label{noisy case}
In this section, our focus is on examining the regularizing nature of the method as the parameter $\delta = (\delta_0, \delta_1, \cdot \cdot \cdot, \delta_{P-1})$ approaches to zero. To keep the things straightforward, we select any positive value $\delta_n = (\delta_{0, n}, \delta_{1, n}, \cdot \cdot \cdot, \delta_{P-1, n})$ and its corresponding noisy data $y^{\delta_n}$ within the set $\mathcal{B}(y^{\dagger}, \delta_n)$. We introduce two variables, namely $K_n = k(\delta_n)$ and $u_n = u_{K_n}^{\delta_n}$.
\begin{prop}
    Assume that all hypotheses of Proposition $\ref{decprop}$ hold and  assumption $(A1)$ is fulfilled. Let, for any uniformly drawn index $i_k \in \{0, 1, \cdot \cdot P-1 \},$ $\{\delta_{i_k, n} \}_{n \in \mathbb{N}}$ be a positive zero sequence. Then, any subsequence of $\{u_n \}_{n\in \mathbb{N}}$ contains a further subsequence that converges weakly almost surely to some $\Bar{u} \in \mathcal{D}(u^{\dagger},\sigma)$ in U. In addition, if $u^{\dagger}$ is the unique solution of $\ref{ststart}$ in $\overline{\mathcal{B}}_U(u^{\dagger}, \sigma)$, then $\{u_n \}_{n \in \mathbb{N}}$ converges weakly to $u^{\dagger}$ almost surely in U.
\end{prop}

\begin{proof}
   See \cite[Proposition 2.4]{clason2019bouligand} for the proof.
\end{proof}
Next, we will show that the SDBLI method \ref{stmain} exhibits asymptotic stability discussed in Definition \ref{def-AS}. To begin, we establish several essential lemmas that will contribute to our main result. Specifically, the following lemma contributes to show (i) of Definition \ref{def-AS}.
\begin{lem}\label{new form}
   Assume that assumptions $(A1)$-$(A5)$ hold. For an arbitrarily chosen starting point $u_0 \in \overline{\mathcal{B}}_U(u^{\dagger}, \sigma)$,  step size $\omega_k \in [\omega, \Omega]$ and uniformly drawn index $i_k \in \{0, 1, \cdot \cdot P-1 \},$ let  $\{\delta_n \}_{n \in \mathbb{N}}$ be a positive zero sequence. Then there exist a subsequence $\{\delta_{n_j}\}_{j \in \mathbb{N}}$ and a sequence $\{\hat{u}_k\}_{k\in \mathbb{N}} \subset \overline{\mathcal{B}}_U(u^{\dagger}, \sigma)$ almost surely such that
   \begin{equation}
       u_{k}^{\delta_{n_j}} \rightarrow \hat{u}_k \hspace{5mm}  \textit{as} \hspace{2mm} j \rightarrow \infty,
   \end{equation}
holds. 
     Moreover,    the sequence $\{\hat{u}_k\}_{k\in \mathbb{N}}$ satisfies
    \begin{equation}\label{rk,sk}
        \hat{u}_{k+1} = \hat{u}_{k} + \omega_{k}G_{i_k}(\hat{u}_k)^{*}(y^{\dagger}_{i_k} - F_{i_k}(\hat{u}_k)) + \lambda_{k}M'_{i_k}(\hat{u}_k)^{*}(y^{\dagger}_{i_k} - M_{i_k}(\hat{u}_k))
        +\omega_{k}r_{i_k,k} + \lambda_{k}s_{i_k,k},
    \end{equation}
       $ \hspace{17mm}\hat{u}_0 = u_0,$ \\
    for some $r_{i_k, k} \in Z_1$ and $s_{i_k, k} \in Z_2$ for all $k \geq 0.$
\end{lem}
\begin{proof}
   Clearly, the sequence $\{{K_n}\}_{n \in \mathbb{N}}$ consists of natural numbers.  Therefore, there exists a subsequence ${\delta_{n_j}}$ for which $K_{n_j}$   increases to infinity as $j$ approaches infinity due to \ref{aprioriforSGD}.\\ 
To show the claim made in the lemma, we employ mathematical induction.  To streamline our notation, we define
    \begin{center}
         $u_{k}^{j} = u_{k}^{\delta_{n_j}}, y_{i_k, k}^{j} = F_{i_k}(u_{k}^{\delta_{n_j}}), z_{i_k,k}^{j} = M_{i_k}(u_{k}^{\delta_{n_j}})$ and $y_{i_k}^j = y_{i_k}^{\delta_{n_j}}.$ 
    \end{center}
       For $k=0$ with $\hat{u}_0 = u_0$,  \ref{(i)} is satisfied  trivially. Without loss of generality,  we assume that $\{\delta_{n_j}\}_{j \in \mathbb{N}}$ itself is a subsequence satisfying $u_{k}^{j} \rightarrow\hat{u}_k$ as $j \rightarrow \infty$ for some $\hat{u}_k \in \overline{\mathcal{B}}_U(u^{\dagger}, \sigma).$ By setting
    $$a_{i_k,k}^j = G_{i_k}(u_{k}^j)^*(y_{i_k}^j - y_{i_k,k}^{j}), \hspace{5mm} a_{i_k,k} = G_{i_k}(u_{k})^*(y_{i_k}^{\dagger} - \hat{y}_{i_k,k}), \hspace{5mm} \zeta_{i_k,k}^j = a_{i_k,k}^j - a_{i_k,k}, \hspace{5mm} $$
    and
    $$b_{i_k,k}^j = M'_{i_k}(u_{k}^j)^*(y_{i_k}^j - z_{i_k,k}^{j}), \hspace{5mm} b_{i_k,k} = M'_{i_k}(u_{k})^*(y_{i_k}^{\dagger} - \hat{z}_{i_k,k}), \hspace{5mm} \xi_{i_k,k}^j = b_{i_k,k}^j - b_{i_k,k}, \hspace{5mm} $$
    with $\hat{y}_{i_k,k} = F_{i_k}(\hat{u}_{k}), \hat{z}_{i_k,k} = M_{i_k}(\hat{u}_{k})$, we have that 
    \begin{align*}
        \zeta_{i_k,k}^{j} &=  G_{i_k}(u_{k}^j)^*(y_{i_k}^j - y_{i_k,k}^{j}) - G_{i_k}(u_{k})^*(y_{i_k}^{\dagger} - \hat{y}_{i_k,k}) \\
        & =\left[ G_{i_k}(u_{k}^j)^*( y_{i_k}^{\dagger} - \hat{y}_{i_k,k} ) - G_{i_k}(u_{k})^*(y_{i_k}^{\dagger} - \hat{y}_{i_k,k}) \right] + G_{i_k}(u_{k}^j)^*(y_{i_k}^j - y_{i_k}^{\dagger} + \hat{y}_{i_k,k}- y_{i_k,k}^{j}) \\
        &= \eta_{i_k,k}^{j} - \eta_{i_k,k} + c_{i_k,k}^j,
    \end{align*}
    where
    $\eta_{i_k,k}^{j} = G_{i_k}(u_{k}^j)^*( y_{i_k}^{\dagger} - \hat{y}_{i_k,k} ), \ \  \eta_{i_k,k} =  G_{i_k}(u_{k})^*(y_{i_k}^{\dagger} - \hat{y}_{i_k,k})\ \ \text{and}  \ \    c_{i_k,k}^j =  G_{i_k}(u_{k}^j)^*(y_{i_k}^j - y_{i_k}^{\dagger} + \hat{y}_{i_k,k}- y_{i_k,k}^{j}) .$\\
    Similarly, we note that
     \begin{align*}
        \xi_{i_k,k}^{j} &=  M'_{i_k}(u_{k}^j)^*(y_{i_k}^j - z_{i_k,k}^{j}) - M'_{i_k}(u_{k})^*(y_{i_k}^{\dagger} - \hat{z}_{i_k,k}) \\
        & =\left[ M'_{i_k}(u_{k}^j)^*( y_{i_k}^{\dagger} - \hat{z}_{i_k,k} ) - M'_{i_k}(u_{k})^*(y_{i_k}^{\dagger} - \hat{z}_{i_k,k}) \right] + M_{i_k}(u_{k}^j)^*(y_{i_k}^j - y_{i_k}^{\dagger} + \hat{z}_{i_k,k}- z_{i_k,k}^{j}) \\
        &= \upsilon_{i_k,k}^{j} - \upsilon_{i_k,k} + d_{i_k,k}^j,
    \end{align*}
    where
    $\upsilon_{i_k,k}^{j} = M'_{i_k}(u_{k}^j)^*( y_{i_k}^{\dagger} - \hat{z}_{i_k,k} ), \ \ \upsilon_{i_k,k} =  M'_{i_k}(u_{k})^*(y_{i_k}^{\dagger} - \hat{z}_{i_k,k})\ \ \text{and}\ \   d_{i_k,k}^j =  M'_{i_k}{u_{k}^j}^*(y_{i_k}^j - y_{i_k}^{\dagger} + \hat{z}_{i_k,k}- z_{i_k,k}^{j}) .$
    To this end, assumption $(A1)$,  boundedness of $\{M_i\}_{i \in \{0,1,2 ... P-1\}}$ together with the fact that $u_{k}^{j} \rightarrow \hat{u}_{k}$ imply that $y_{i_k,k}^{j} \rightarrow \hat{y}_{i_k,k}$ and $z_{i_k,k}^{j} \rightarrow \hat{z}_{i_k,k}$. From this and assumption $(A2)$, we further deduce that
    \begin{equation*}\label{ex1}
        c_{i_k,k}^{j} \hspace{2mm} \textit{and} \hspace{2mm} d_{i_k,k}^{j} \rightarrow 0 \in U \textit{as} \hspace{2mm} j \rightarrow \infty.
    \end{equation*}
       Using  assumption $(A5)$, we further see that the sequences $\{\eta_{i_k,k}^j\}_{j \in \mathbb{N}},\{\upsilon_{i_k,k}^j\}_{j \in \mathbb{N}} $ and hence  $\{\eta_{i_k,k}^j - \eta_{i_k,k}\}_{j \in \mathbb{N}} ,\{\upsilon_{i_k,k}^j -\upsilon_{i_k,k}\}_{j \in \mathbb{N}} $  are bounded in $Z_1, Z_2,$ respectively. Since $Z_1, Z_2 \hookrightarrow U$ compactly, there exist $r_{i_k,k} \in Z_1$ and $s_{i_k,k} \in Z_2$ and a subsequence of $\{\delta_{n_j}\}_{j \in \mathbb{N}}$, denoted in the same way, such that
   \begin{equation}\label{rn}
       \eta_{i_k,k}^j - \eta_{i_k,k} \rightarrow r_{i_k,k} \hspace{5mm} \text{and} \hspace{5mm} \upsilon_{i_k,k}^j -\upsilon_{i_k,k} \rightarrow s_{i_k,k} \in U \textit{as} \hspace{2mm} j \rightarrow \infty.
   \end{equation}
  By definition, we observe that
   \begin{align*}
       u_{k+1}^j &= u_{k}^j + \omega_{k}G_{i_k}(u_{k}^j)^*(y_{i_k}^j - y_{i_k,k}^j) + \lambda_{k}M'_{i_k}(u_{k}^j)^*(y_{i_k}^j - z_{i_k,k}^j) \\
       &= u_{k}^j +\omega_{k}a_{i_k,k}^j + \lambda_{k}b_{i_k,k}^j \\
       &= u_{k}^{j} + \omega_{k}a_{i_k,k} + \omega_k(\eta_{i_k,k}^{j} - \eta_{i_k,k}) + \omega_{k}c_{i_k,k}^j +
        \lambda_{k}b_{i_k,k} + \lambda_k(\upsilon_{i_k,k}^{j} - \upsilon_{i_k,k}) + \lambda_{k}d_{i_k,k}^j.
   \end{align*}
   Taking limit $j \rightarrow \infty$, and the fact that $u_{k}^{j} \rightarrow \hat{u}_k$ along with \ref{rn} to reach at
   \begin{align*}
       u_{k+1}^j &\rightarrow \hat{u}_k + \omega_{k}a_{i_k,k} + \omega_{k}r_{i_k,k} + \lambda_{k}b_{i_k,k} + \lambda_{k}s_{i_k,k} \\
      &= \hat{u}_k + \omega_{k}G_{i_k}(u_{k})^*(y_{i_k}^{\dagger} - \hat{y}_{i_k,k}) + \omega_{k}r_{i_k,k} + \lambda_{k} M'_{i_k}(u_{k})^*(y_{i_k}^{\dagger} - \hat{z}_{i_k,k}) + \lambda_{k}s_{i_k,k}.
   \end{align*}
  By setting $\hat{u}_{k+1} = \hat{u}_k + \omega_{k}G_{i_k}(u_{k})^*(y_{i_k}^{\dagger} - \hat{y}_{i_k,k}) + \omega_{k}r_{i_k,k} + \lambda_{k} M'_{i_k}(u_{k})^*(y_{i_k}^{\dagger} - \hat{z}_{i_k,k}) + \lambda_{k}s_{i_k,k}, $ we get  \ref{rk,sk}    for $k+1$. Also, as   $u_{k+1}^j \in \overline{\mathcal{B}}_U(u^{\dagger}, \sigma)$ almost surely for all $j \in \mathbb{N}$, this implies that $\hat{u}_{k+1} \in \overline{\mathcal{B}}_U(u^{\dagger}, \sigma)$ almost surely.   This completes the proof.
\end{proof}
The following lemma will give several estimates that will be incorporated later on in Lemma \ref{big} in showing (ii) of Definition \ref{def-AS}.
\begin{lem}\label{norm}
   Under the conditions of Lemma $\ref{new form}$,   the sequences $\{\hat{u}_k\}_{k \in \mathbb{N}}, \{r_{i_k, k}\}$, and $\{s_{i_k, k}\}$ defined in $\ref{rk,sk}$ adhere to the subsequent estimates.\\
     For all $k \in \mathbb{N}$ and $i_k \in \{0,1,2,... P-1\}$, we have 
    \begin{enumerate}
        \item [(i)]$\|r_{i_k,k}\| \leq 2L_F\|y_{i_k}^{\dagger} - \hat{y}_{i_k,k}\|,$
        \item [(ii)] $\|s_{i_k,k}\| \leq 2L_M C_{M}^0,$
        \item [(iii)]$(r_{i_k,k}, \hat{u}_k - \hat{u}) \leq (-1 + \mu)\|y_{i_k}^{\dagger} - \hat{y}_{i_k,k}\|^{2} - (y^{\dagger} - \hat{y}_k, G_{i_k}(\hat{u}_k)(\hat{u}_k - \hat{u})),$
        \item  [(iv)]$(s_{i_k,k}, \hat{u}_k - \hat{u}) \leq 2L_M C_{M}^0 \sigma ,$ 
        \item [(v)] $|(r_{i_k,k}, \hat{u}_m - \hat{u})| \leq 2(1+\mu )\|y_{i_k}^{\dagger} - \hat{y}_{i_k,k}\|\left[ \|y_{i_k}^{\dagger} - \hat{y}_{i_k,k}\| + \|\hat{y}_{i_k,m} - \hat{y}_{i_k,k}\|\right] \hspace{2mm} \forall \hspace{1mm} m \geq 0.$
    \end{enumerate}
    Here $L_F, L_M$ are from assumption $(A2)$ and $\hat{u} \in \mathcal{D}(u^{\dagger}, \sigma).$
\end{lem}
\begin{proof}
 From assumption $(A2)$, by employing the same notations   utilized in Lemma \ref{new form}, we deduce that
 \begin{center}
     $\|\eta_{i_k,k}^j - \eta_{i_k,k}\|=\|G_{i_k}(u_{k}^j)^*( y_{i_k}^{\dagger} - \hat{y}_{i_k,k} ) - G_{i_k}(u_{k})^*(y_{i_k}^{\dagger} - \hat{y}_{i_k,k})\| \leq 2L_F\|y_{i_k}^{\dagger} - \hat{y}_{i_k,k}\|,$
     \end{center}
     and
     \begin{center}
      $\|\upsilon_{i_k,k}^j - \upsilon_{i_k,k}\|=\|M'_{i_k}(u_{k}^j)^*( y_{i_k}^{\dagger} - \hat{z}_{i_k,k} ) - M'_{i_k}(u_{k})^*(y_{i_k}^{\dagger} - \hat{z}_{i_k,k})\| \leq 2L_M\|y_{i_k}^{\dagger} - \hat{z}_{i_k,  k}\|.$
 \end{center}
By incorporating \ref{rn}, we attain that  
 $$\|r_{i_k,k}\| = \lim_{j \rightarrow \infty}\|\eta_{i_k,k}^j - \eta_{i_k,k}\|\leq 2L_F\|y_{i_k}^{\dagger} - \hat{y}_{i_k,k}\|,$$ 
 and  
 $$\|s_{i_k,k}\| = \lim_{j \rightarrow \infty}\|\upsilon_{i_k,k}^j - \upsilon_{i_k,k}\|\leq 2L_M\|y_{i_k}^{\dagger} - \hat{z}_{i_k,k}\| = 2L_M C_{M}^{0}.$$
 This proves assertions (i) and (ii). 
 The proofs of assertions (iii) and (v) closely resemble to the one presented in \cite[Lemma 2.7]{clason2019bouligand} and are therefore omitted here. Finally, we prove assertion (iv). By employing Cauchy-Schwartz inequality, we get
 \begin{align*}
     (s_{i_k,k}, \hat{u}_k - \hat{u}) \leq \|s_{i_k,k}\|\|\hat{u}_k - \hat{u}\| = 2L_M C_{M}^0 \sigma.
 \end{align*}
    This concludes the proof.
\end{proof}
\begin{lem}\label{big}
    Suppose that assumption $(A1)$-$(A5)$ hold. Let $\Omega$ and $\omega$ satisfy  $\Omega \geq \omega >0$ and
\begin{equation}\label{stCh}
     C_H = 2\left[\Omega(-1 + \mu + 10\Omega L_{F}^2) + C_{\lambda}L_M C_M^0\left(3\sigma +10C_{\lambda}L_M C_M^0  \right) \right]\leq 0.
     \end{equation}
Moreover, let us assume that there exist a constant $C_{\lambda}$ such that 
\begin{equation}\label{C00}
    \lambda_k \leq C_{\lambda} \|y_{i_k}^{\dagger} - \hat{y}_{i_k,k}\|^2 \hspace{5mm} \forall \hspace{2mm} i_k \in \{0,1,2... P-1\}.
\end{equation}
Assuming the step size $\omega_k \in [\omega, \Omega],$ starting point $ u_0 \in \overline{\mathcal{B}}_U(u^{\dagger}, \sigma)$ is arbitrarily chosen and  the sequence $\{\hat{u}_k\}_{k \in \mathbb{N}}$  is defined by $\ref{rk,sk}$ and it satisfies assertions $(i)$–$(v)$   
  of Lemma $\ref{norm}$. Then $\{\hat{u}_k\}_{k \in \mathbb{N}}$ converges almost surely to some $\Bar{u} \in \mathcal{D}(u^{\dagger}, \sigma)$ as $k  \rightarrow \infty.$ 
\end{lem}
\begin{proof}
   By definition of $\hat{u}_{k+1}$, we can write
    \begin{equation}\label{CSI}
  \|\hat{u}_{k+1} - u^{\dagger}\|^2 -  \|\hat{u}_{k} - u^{\dagger} \|^2 = 2\left(\hat{u}_k - u^{\dagger}, \hat{u}_{k+1} - \hat{u}_{k}\right) + \| \hat{u}_{k+1} - \hat{u}_{k}\|^2. 
\end{equation} Next, using \ref{rk,sk} we  solve right hand side of  \ref{CSI}  in parts. We note that
\begin{align*}
    \left(\hat{u}_k - u^{\dagger}, \hat{u}_{k+1} - \hat{u}_{k}\right) &=  \left(\hat{u}_k - u^{\dagger},\omega_{k}G_{i_k}(\hat{u}_k)^{*}(y_{i_k}^{\dagger} - F_{i_k}(\hat{u}_k)) \right) \\
    &  \hspace{10mm}+ \left(\hat{u_k} - u^{\dagger}, \lambda_{k}M'_{i_k}(\hat{u}_k)^{*}(y_{i_k}^{\dagger} - M_{i_k}(\hat{u}_k))
        +\omega_{k}r_{i_k,k} + \lambda_{k}s_{i_k,k} \right)  \\
        &= \omega_k(G_{i_k}(\hat{u}_k)(\hat{u_k} - u^{\dagger}), (y_{i_k}^{\dagger} - \hat{y}_{i_k,k})) + \omega_k(\hat{u}_k - u^{\dagger}, r_{i_k,k}) \\ & \hspace{10mm} + \lambda_k(M'_{i_k}(\hat{u}_k)(\hat{u}_k - u^{\dagger}), (y_{i_k}^{\dagger} - \hat{z}_{i_k,k})) + \lambda_k(\hat{u}_k - u^{\dagger}, s_{i_k,k}) \\
        &\leq \omega_k(-1 + \mu)\|y_{i_k}^{\dagger} - \hat{y}_{i_k}\|^2 + \sigma\lambda_{k}L_{M}C_{M}^0 + \sigma \lambda_{k} \|s_{i_k,k}\|.
\end{align*}
Incorporating assertion (ii) of Lemma \ref{norm} and  \ref{C00} in the last inequality to further obtain
\begin{equation}\label{E3rh}
  \left(\hat{u_k} - u^{\dagger}, \hat{u}_{k+1} - \hat{u}_{k}\right)   \leq [\omega_k(-1 + \mu) + 3\sigma C_{\lambda} L_M C_{M}^0]\|y_{i_k}^{\dagger} - \hat{y}_{i_k,k}\|^{2}.
\end{equation}
Similarly, for the second term in the right hand side of \ref{CSI}, we have 
\begin{align*}
    \| \hat{u}_{k+1} - \hat{u}_{k}\|^2 &= \|\omega_{k}G_{i_k}(\hat{u}_k)^{*}(y_{i_k}^{\dagger} - F_{i_k}(\hat{u}_k)) + \lambda_{k}M'_{i_k}(\hat{u}_k)^{*}(y_{i_k}^{\dagger} - M_{i_k}(\hat{u}_k))
        +\omega_{k}r_{i_k,k} + \lambda_{k}s_{i_k,k}\|^2 \\ 
        & \leq 2\left(\omega_{k}^2(\|G_{i_k}(\hat{u}_k)^{*}(y_{i_k}^{\dagger} - F_{i_k}(\hat{u}_k)) + r_{i_k,k} \|^2) + \lambda_{k}^2(\|M'_{i_k}(\hat{u}_k)^{*}(y_{i_k}^{\dagger} - M_{i_k}(\hat{u}_k))
         + s_{i_k,k}\|^2 )\right) \\
         & \leq 4\left(\omega_{k}^2(\|G_{i_k}(\hat{u}_k)^{*}(y_{i_k}^{\dagger} - F_{i_k}(\hat{u}_k))\|^2 + \|r_{i_k,k} \|^2) \right) \\
        & \hspace{20mm} + 4\left(\lambda_{k}^2(\|M'_{i_k}(\hat{u}_k)^{*}(y_{i_k}^{\dagger} - M_{i_k}(\hat{u}_k))\|^2
         + \|s_{i_k,k}\|^2 )\right).
\end{align*}
Incorporating Lemma \ref{norm} and  \ref{C00} in the last inequality to attain
\begin{equation}\label{5omeg}
   \| \hat{u}_{k+1} - \hat{u}_{k}\|^2   \leq 20(\omega_{k}^2L_{F}^2 + C_{\lambda}^2 L_{M}^2{C_{M}^0}^2 )\|y_{i_k}^{\dagger} - \hat{y}_{i_k,k}\|^2.
\end{equation}
Inserting  \ref{E3rh}, \ref{5omeg} in \ref{CSI} to get 
\begin{align*}
    \|\hat{u}_{k+1} - u^{\dagger}\|^2 -  \|\hat{u_{k}} - u^{\dagger} \|^2 
    &\leq 2\left[\omega_k(-1 + \mu) + 3\sigma C_{\lambda} L_M C_{M}^0 \right]\|y_{i_k}^{\dagger} - \hat{y}_{i_k,k}\|^2 \\ 
    & \hspace{10mm}+ 20(\omega_{k}^2L_{F}^2 + C_{\lambda}^2 L_{M}^2{C_{M}^0}^2 ) \|y_{i_k}^{\dagger} - \hat{y}_{i_k,k}\|^2 
\end{align*}
\begin{equation*}
     \hspace{10mm} \leq 2\left[\Omega(-1 + \mu + 10\Omega L_{F}^2) + C_{\lambda} L_M C_M^{0}\left(3\sigma +10C_{\lambda}L_M C_M^{0} \right) \right]\|y_{i_k}^{\dagger} - \hat{y}_{i_k,k}\|^2 
\end{equation*}
for all $k \geq 0.$
As a result, due to the measurability of $\hat{u}_k$ concerning $\mathcal{F}_k,$  and the Cauchy-Schwarz inequality, we have
$$\mathbb{E}[\|\hat{u}_{k+1} - u^{\dagger}\|^2 -  \|\hat{u}_{k} - u^{\dagger} \|^2 | \mathcal{F}_k]\hspace{80mm}$$ $$ \leq  2\left[\Omega(-1 + \mu + 10\Omega L_{F}^2) + C_{\lambda} L_M C_M^{0}\left(3\sigma +10C_{\lambda} L_M C_M^{0}  \right) \right]\|y^{\dagger} - \hat{y}_{_k}\|^2.$$
Finally,  taking the complete conditional yields
\begin{multline}
     \mathbb{E}[\|\hat{u}_{k+1} - u^{\dagger}\|^2] -  \mathbb{E}[\|\hat{u}_{k} - u^{\dagger} \|^2 ] \leq \\ \label{st10omg}  2\left[\Omega(-1 + \mu + 10\Omega L_{F}^2) + C_{\lambda} L_M C_M^{0}\left(3\sigma +10C_{\lambda} L_M C_M^{0}  ]\right) \right] \mathbb{E}[\|y^{\dagger} - \hat{y}_{_k}\|^2].
\end{multline}
Consequently, we have 
\begin{equation}\label{sumsum}
    \sum_{k \geq 0} \mathbb{E}[\|y^{\dagger} - \hat{y}_k\|^2] \leq \frac{1}{C_H}\|\hat{u}_{0} - u^{\dagger}\|^2 < \infty,
\end{equation}
where $C_H$ is same as in \ref{stCh}. Further, \ref{st10omg} also indicates that the sequence $\{\mathbb{E}[\|\hat{e}_k\|^2]\}_{k \in \mathbb{N}}$ with $\hat{e}_{k} = u^{\dagger} - \hat{u}_{k}$ is monotonically decreasing. Therefore, it follows that $\lim_{k \rightarrow \infty} \mathbb{E}[\|\hat{e}_k\|^2]^{\frac{1}{2}}:= \hat{\pi} \geq 0. $ We emphasize that our interim aim is to show that the sequence $\{u_k\}_{k\geq 1}$ is indeed almost surely a Cauchy sequence. To show this, for any $m ,l \in \mathbb{N}$ with $m \leq l,$ we choose
\begin{equation}
    n \in \arg \min_{m \leq s \leq l} \mathbb{E}[\|y^{\dagger} - \hat{y}_s\|^2].
\end{equation}
Because of the inequality 
\[
\mathbb{E} [\| \hat{e}_m - \hat{e}_l \|^{2}]^{1/2} \leq \mathbb{E} [\| \hat{e}_m - \hat{e}_n \|^{2}]^{1/2} + \mathbb{E} [\| \hat{e}_n - \hat{e}_\ell \|^{2}]^{1/2}
\]
and the identities
\begin{align}\label{E1}
\mathbb{E} [\| \hat{e}_m - \hat{e}_{n} \|^{2}] &= 2\mathbb{E} [( \hat{e}_{n} - \hat{e}_m , \hat{e}_n ) ] + \mathbb{E} [\| \hat{e}_m \|^{2}] - \mathbb{E} [\| \hat{e}_n \|^{2}], \\ \label{E2}
\mathbb{E} [\| \hat{e}_n - \hat{e}_l \|^{2}] &= 2\mathbb{E} [( \hat{e}_n - \hat{e}_l, \hat{e}_n ) ] + \mathbb{E} [\| \hat{e}_l \|^{2}] - \mathbb{E} [\| \hat{e}_n \|^{2}],
\end{align}
it suffices to prove that both $ \mathbb{E} [\| \hat{e}_m - \hat{e}_{n} \|^{2}] $ and $  \mathbb{E} [\| \hat{e}_n - \hat{e}_l \|^{2}]  $ tend to zero as $ l \geq n \geq m \rightarrow \infty$.

For $ l \geq n \geq m \rightarrow \infty$, the last two terms on the right-hand sides of  \ref{E1}  and  \ref{E2}  tend to $\hat{\pi}^2 - \hat{\pi}^2 = 0$, by the monotone convergence of $\mathbb{E} [\| e_k \|^{2}]$ to $\hat{\pi}^2$.

Next, we show that the term $\mathbb{E} [( \hat{e}_n - \hat{e}_m, \hat{e}_n ) ]$ also tends to zero as $m \rightarrow \infty.$ By the definition of $\hat{u}_k,$ we have
\begin{equation}\label{lemmaain}
    (\hat{e}_n - \hat{e}_m, \hat{e}_n) = \sum_{k=m}^{n-1}(\hat{e}_{k+1} - \hat{e}_k, \hat{e}_n).
\end{equation}
From  \ref{rk,sk}, we obtain 
\begin{center}
    $\hat{e}_{k+1} - \hat{e}_k =  -\omega_{k}G_{i_k}(\hat{u_k})^{*}(y_{i_k}^{\dagger} - F_{i_k}(\hat{u_k})) - \lambda_{k}M'_{i_k}(\hat{u_k})^{*}(y_{i_k}^{\dagger} - M_{i_k}(\hat{u_k}))
        -\omega_{k}r_{i_k,k} - \lambda_{k}s_{i_k,k}$
\end{center}
and hence
\begin{align*}
(\hat{e}_{k+1} - \hat{e}_k, \hat{e}_n) &= -\omega_k(y_{i_k}^{\dagger} - \hat{y}_{i_k,k}, G_{i_k}(\hat{u}_k)\hat{e}_n) - \omega_k(r_{i_k,k}, \hat{e}_n)\\
&\hspace{25mm}-\lambda_k(y_{i_k}^{\dagger} - \hat{z}_{i_k,k}, M'_{i_k}(\hat{u}_k)\hat{e}_n) - \lambda_k(s_{i_k,k}, \hat{e}_n) \\
&= \omega_k(y_{i_k}^{\dagger} - \hat{y}_{i_k,k}, G_{i_k}(\hat{u}_k)(\hat{u}_n - u^{\dagger}) + \omega_k(r_{i_k,k}, (\hat{u}_n - u^{\dagger}))\\ &\hspace{25mm}+\lambda_k(y_{i_k}^{\dagger} - \hat{z}_{i_k,k}, M'_{i_k}(\hat{u}_k)(\hat{u}_n - u^{\dagger})) + \lambda_k(s_{i_k,k}, (\hat{u}_n - u^{\dagger})).
\end{align*}
It follows that
\begin{align*}
    | (\hat{e}_{k+1} - \hat{e}_k, \hat{e}_n)  | \leq \omega_k\|y_{i_k}^{\dagger} - \hat{y}_{i_k,k}\| \|G_{i_k}(\hat{u}_k)(\hat{u}_n - u^{\dagger})\| + \omega_k|(r_{i_k,k}, (\hat{u}_n - u^{\dagger}))|
    \end{align*}
    \begin{align} \label{1}
         \hspace{50mm}+\lambda_k\|y_{i_k}^{\dagger} - \hat{z}_{i_k,k}\| \|M'_{i_k}(\hat{u}_k)(\hat{u}_n - u^{\dagger}))\| + \lambda_k|(s_{i_k,k}, (\hat{u}_n - u^{\dagger}))|.
    \end{align}
    Also it is easy to deduce that 
    \begin{equation}\label{2}
        \|G_{i_k}(\hat{u}_k)(\hat{u}_n - u^{\dagger})\| \leq 3(1+\mu)\|y_{i_k}^{\dagger} - \hat{y}_{i_k,k}\|.
    \end{equation}
    On the other hand, assertion $(v)$ of Lemma \ref{norm} implies that 
    \begin{equation}\label{3}
        |\left(r_{i_k,k}, \hat{u}_n - u^{\dagger} \right)| \leq 6(1+\mu)\|y_{i_k}^{\dagger} - \hat{y}_{i_k,k}\|^2.
    \end{equation}
    Similarly,  we can derive that 
    \begin{equation}\label{4}
        \|M'_{i_k}(\hat{u}_k)(\hat{u}_n - u^{\dagger})\| \leq \|M'_{i_k}(\hat{u}_k)\| \|(\hat{u}_n - u^{\dagger})\| = L_M \sigma,
    \end{equation}
    and 
    \begin{equation} \label{5}
        \left|\left(s_{i_k,k}, \hat{u}_n - u^{\dagger} \right)\right| \leq \|s_{i_k,k}\| \|\hat{u}_n - u^{\dagger}\| = 2L_M C_{M}^0 \sigma.
    \end{equation}
    Combining \ref{1}-\ref{5} with the inequality in \ref{lamdd} to reach at 
\begin{equation*}
     | (\hat{e}_{k+1} - \hat{e}_k, \hat{e}_n)  | \leq 3\left[3\omega_k(1+\mu) + \sigma L_M  C_{\lambda}C_{M}^0 \right ] \|y_{i_k}^{\dagger} - \hat{y}_{i_k,k}\|^2.
\end{equation*}
As a result of $\hat{u}_k$ being measurable with respect to $\mathcal{F}_k,$ we have 
 \begin{equation*}
     |\mathbb{E}[(\hat{e}_{k+1} - \hat{e}_k, \hat{e}_n)| \mathcal{F}_k ] | \leq 3\left[3\omega_k(1+\mu) +  \sigma L_M C_{\lambda}C_{M}^0 \right ] \|y^{\dagger} - \hat{y}_k\|^2.
\end{equation*}
    Thus, by considering the full conditional, we get 
    \begin{equation}\label{mulm}
        | \mathbb{E}[(\hat{e}_{k+1} - \hat{e}_k, \hat{e}_n)]  | \leq 3\left[3\omega_k(1+\mu) + \sigma L_M  C_{\lambda}C_{M}^0 \right ] \mathbb{E}[\|y^{\dagger} - \hat{y}_k\|^2].
    \end{equation}
    Using  \ref{mulm}  in \ref{lemmaain}, we attain that 
    \begin{align*}
       | \mathbb{E}[(\hat{e}_n - \hat{e}_m, \hat{e}_n)] |  &\leq  3 \sum_{k=m}^{n-1}\left[3\omega_k(1+\mu) + \sigma L_M  C_{\lambda}C_{M}^0  \right ] \mathbb{E}[\|y^{\dagger} - \hat{y}_k\|^2]\\
        &\leq  3 \left[3\Omega(1+\mu) +  \sigma L_M C_{\lambda}C_{M}^0 \right ]  \sum_{k=m}^{n-1}\mathbb{E}[\|y^{\dagger} - \hat{y}_k\|^2].
    \end{align*}
    Similarly, we can derive that 
    \begin{align*}
        |\mathbb{E}(\hat{e}_n - \hat{e}_l, \hat{e}_n)] |  &\leq \sum_{k=n}^{l-1} \left[3\omega_k(1+\mu) + \sigma L_M  C_{\lambda}C_{M}^0\right ] \mathbb{E}[\|y^{\dagger} - \hat{y}_k\|^2]\\
        &\leq   3 \left[3\Omega(1+\mu) + \sigma L_M  C_{\lambda}C_{M}^0 \right ]  \sum_{k=n}^{l-1} \mathbb{E}[\|y^{\dagger} - \hat{y}_k\|^2].
    \end{align*}
  These two estimates along with \ref{sumsum}  lead to the conclusion that the right hand side of  \ref{E1}  and  \ref{E2}  tends to zero as $m \rightarrow \infty.$ Consequently, we can now assert that the sequence ${\hat{u}_k}$ is almost surely a Cauchy sequence in the space $U$. As a direct consequence, there exists a limit $\overline{u}$ in $\overline{\mathcal{B}}_U(u^{\dagger}, \sigma)$ such that $\hat{u}_k$ converges to $\overline{u}$ almost surely. This convergence implies that $F(\hat{u}_k)$ approaches $F(\overline{u})$ almost surely as $k$ tends to infinity. Further, from \ref{sumsum}, we have $\mathbb{E}[y^{\dagger} - F(\hat{u}_k)] \rightarrow 0$ as $k$ approaches infinity.\\
Consequently, we have  that $y^{\dagger} =F(\overline{u})$ almost surely, which implies that $\overline{u}$ belongs to $\mathcal{D}(u^{\dagger}, \sigma)$ almost surely. This completes the proof.
\end{proof}
We have thus shown the following outcome.
\begin{cor}\label{stcoro}
 Under the  assumptions of Lemmas $\ref{new form}$ and $\ref{big}$,  the  SDBLI method $\ref{stmain}$  with a-priori stopping rule is asymptotically stable for a starting point $u_0 \in \overline{\mathcal{B}}_{U}(u^{\dagger},\sigma)$ and the step sizes $\{\omega_k\}_{k \in \mathbb{N}} \subset[\omega, \Omega]$ for $\Omega \geq \omega > 0$. 
\end{cor}
    We are now well-equipped to support our main finding.
     
    \begin{thm}
Let assumptions of Proposition 1 along with    $\ref{stCh}$ hold. Further, let $\{\delta_n \}_{n \in \mathbb{N}}$ be a positive zero sequence. Let the starting point  $u_0 \in \overline{\mathcal{B}}_{U} (u^{\dagger},\sigma)$ and the data-driven factor $\lambda_k$,  step sizes $\{\omega_k\}_{k \in \mathbb{N}} \subset[\omega, \Omega]$ be arbitrary and  let the stopping index $k(\delta)$ be chosen according to  an \text{a priori} stopping rule. Then, any subsequence of $\{u_{K_n}^{\delta_n}\}_{n \in \mathbb{N}}$ contains a subsequence that converges strongly almost surely to an element of $\mathcal{D}(u^{\dagger},\sigma)$, where $K_n=k(\delta_n)$. Furthermore, if $u^{\dagger}$ is the unique solution, then
\begin{center}
    $u_{K_n}^{\delta_n} \rightarrow u^{\dagger}$ almost surely in $U$ as $n \rightarrow \infty$.
\end{center}
    \end{thm}
    \begin{proof}
 Let $\{\delta_{n_j}\}_{j\in\mathbb{N}}$ be an arbitrary subsequence of $\{\delta_n\}_{n\in\mathbb{N}}$. By virtue of Corollary \ref{stcoro}, there exists a sequence $\{\hat{u}_k\} \subset \overline{\mathcal{B}}_U(u^\dagger, \sigma)$ and a subsequence of $\{\delta_ {n_j}\}_{j\in\mathbb{N}}$, denoted in the same way, satisfying conditions (i)-(ii) in Definition \ref{def-AS}.\\
For each  $(\delta_{n_j})$, we denote $K_{n_j}=k(\delta_{n_j})  \rightarrow \infty$ as $j \rightarrow \infty$ as the stopping index. Further we may assume without loss of generality that $\{K_{n_j}\}_{j \in \mathbb{N}}$ is monotonically increasing.\\
The condition (i) and (ii) of Definition \ref{def-AS} confirms the existence of   $\hat{u} \in \mathcal{D}(u^\dagger, \sigma)$ that, together with $\{\delta_{n_j}\}_{j\in\mathbb{N}}$ and $\{\hat{u}_k\}_{k\in\mathbb{N}}$, satisfies
\begin{equation}\label{stconlast}
    u_{k}^{\delta_{n_j}} \rightarrow \hat{u}_k \text{ almost surely  as } j \rightarrow \infty, \text{ for } 0 \leq k \leq K_{n_j} \text{ with all } j \text{ large enough}
\end{equation}
and
\begin{equation}\label{stconIIlast}
    \hat{u}_k \rightarrow \hat{u} \text{ almost surely  as } k \rightarrow \infty. 
\end{equation}
From \ref{stconIIlast}, for each $\varepsilon > 0$, there exists an integer $k^*$ such that
\[
\|\hat{u}_{k^*} - \hat{u}\|  < \frac{\varepsilon}{2} \ \ \text{almost surely}.
\]
It also follows from  \ref{stconlast} and the fact that $K_{n_j}$ tends increasingly to infinity as $j \rightarrow \infty$ that an $\overline{j} \in \mathbb{N}$ exists such that

\[
k^* \leq K_{n_j} \text{ and } \|u_{k^*}^{\delta_{n_j}} - \hat{u}_{k^*}\|  < \frac{\varepsilon}{2} \text{ for all } j \geq \overline{j}\ \ \text{almost surely}.
\]Consequently, Proposition \ref{decprop}   implies that
\[
\|u^{\delta_{n_j}}_{K_{n_j}} - \hat{u}\|  \leq \|u^{\delta_{n_j}}_{k^*} - \hat{u}\|  \leq \|u^{\delta_{n_j}}_{k^*} - \hat{u}_{k^*}\|  + \|\hat{u}_{k^*} - \hat{u}\|  < \varepsilon \text{ for all } j \geq \overline{j}
\]almost surely. 
We thus obtain that
\[
\lim_{j\rightarrow\infty}\|u^{\delta_{n_j}}_{K_{n_j}} - \hat{u}\|_U = 0 \ \text{almost surely as claimed.}
\]
This completes the proof.    
    \end{proof}


   \section{Example}\label{NE}
 The   objective of this section is to discuss a system of inverse problems on which our results are applicable. 
\subsection*{System of inverse source problems}
We investigate a system of  inverse source problems that are associated with the system of elliptic partial differential equations. In this context, we consider a region denoted as $\mathcal{I} \subset \mathbb{R}^r$, where $r=2$ or $r=3$. $\mathcal{I}$ is assumed to be open and bounded, and its boundary $\partial \mathcal{I}$, is Lipschitz.
 For $u\in L^2(\mathcal{I})$ and $y_{i}^+(t):=\max(y_{i}(t), 0)$ for almost every $t$ in $\mathcal{I}$, consider the following system of semilinear  equations 
 \begin{equation}\label{nfexample}
\begin{cases}
           -\Delta y_{i} + y_{i}^+= u \ \text{in}\ \mathcal{I}, \quad & \, i = 0,1,... P-1, \\
          \hspace{2mm} y_{i} \big|_{\partial \mathcal{I}}=0. \quad & \, \\
     \end{cases}
\end{equation}
We refer to the works of \cite{kikuchi1984finite, rappaz1984approximation} for detailed information on the  various models that involve  the equations  \ref{nfexample}. Let, for each $i$, $\mathcal{F}_{i}$ be  the solution operator of  \ref{nfexample}, where   $$\mathcal{F}_{i}:L^2(\mathcal{I})\to H^1_0(\mathcal{I})\cap C(\bar{\mathcal{I}}).$$

 By utilizing the findings presented in \cite{ christof2017optimal, clason2019bouligand}, it can be demonstrated that the operator $\mathcal{F}_{i}$ and its Bouligand subdifferential fulfill Assumption $3.1$.

More specifically, for each $i \in \{0, 1, 2,... P-1\},$ the following statements hold:
\begin{itemize}
    \item [(i)] The function $\mathcal{F}_i$ exhibits global Lipschitz continuity in $L^2(\mathcal{I})$, which is demonstrated in \cite[Proposition 2.1]{christof2017optimal}. Furthermore, the complete continuity of the operator $\mathcal{F}_i:L^2(\mathcal{I})\to L^2(\mathcal{I})$ follows from \cite[Lemma 3.2]{clason2019bouligand}. 
    \item [(ii)] For any $u, h\in L^2(\mathcal{I})$, the directional differentiability of $\mathcal{F}_i : L^2(\mathcal{I})\to H_0^1(\mathcal{I})$ can be deduced from \cite[Theorem 2.2]{christof2017optimal}. However, in general case, $\mathcal{F}_i$ does not exhibit G\^ateaux differentiability, see \cite[Proposition 3.4]{clason2019bouligand} for a more comprehensive explanation of this aspect.
    \item [(iii)]   As mentioned in point (i) above, for all $u \in L^{2}(\mathcal{I})$, the mapping $G_{i}(u) \in \partial_{B} \mathcal{F}_{i}(u)$ is uniformly bounded. Additionly, the operator ${G_{i}(u)} \in L(L^{2}(\mathcal{I}), L^{2}(\mathcal{I}))$ is self adjoint. Moreover, with the help of \cite{christof2017optimal}, we know that there exist constants $L$ and $\hat{L}$ such that $\|G_{i}(u)\|_{\mathbb{B}(L^{2}(\mathcal{I}), H_{0}^{1}(\mathcal{I}))} \leq \hat{L}$ and $ \|G_{i}(u)\|_{\mathbb{B}(L^{2}(\mathcal{I}), L^{2}(\mathcal{I}))} \leq L.$ To understand the detailed characterization of the Bouligand subderivative, one may refer \cite[Proposition 3.16]{christof2017optimal}.  
    \item[(iv)]   The modified  tangential cone condition $\ref{tangential cone}$ is fulfilled by $\mathcal{F}_{i}$ (cf. \cite[Proposition 3.16]{clason2019bouligand}).
\end{itemize}
Thus, Assumption 3.1 is satisfied by the system of inverse source problems under consideration. Next, we discuss a method to construct $M_i$ in \ref{stmain1}.
 \subsection*{Construction of $M_i$}
 We know that the operator $M_i$ of the data-driven term is a bounded linear operator. In this work, the method to construct $M_i$ is adopted from \cite{aspri2020data, tong2023data}. Let us provide a concise overview of this approach.

Assuming the availability of data pairs $(u^{(l)}, y^{(l)}_i) \in U\times Y_i$ for $l = 1, \ldots, N$, the operator $M_i$ is formulated to minimize the functional
\[
\min_{S_i \in \mathbb{B}(U,Y_{i})}\left[ \frac{1}{2} \sum_{l=1}^{N} \left\|S_i u^{(l)} - y^{(l)}_i\right\|^2 \right].
\]
When dealing with the discrete case, we consider
\[
(u^{(l)}, y_{i}^{(l)}) \in \mathbb{R}^{\mathcal{N}_{u}} \times \mathbb{R}^{\mathcal{N}_{y}}, \ \text{for} \ l = 1,\ldots, N.
\]
Then the matrix $M_{i} \in \mathbb{R}^{\mathcal{N}_{y}} \times \mathbb{R}^{\mathcal{N}_{u}}$ is established as the linear mapping such that 
\[
M_{i}V = \mathcal{Y}_{i}, 
\]
where $V \in \mathbb{R}^{\mathcal{N}_u \times N}$ and $\mathcal{Y}_i \in \mathbb{R}^{\mathcal{N}_y \times N}$ contain the data $u^{(l)}$ and $y^{(l)}_i$ columnwise, respectively. Consequently, one can derive the matrix $M_i$ by employing singular value decomposition  on the matrix $V$, i.e.,
\[
M_{i} = \mathcal{Y}_{i}V ^{\dagger}.
\]
Thus, we  get $M_i$ as a bounded linear operator.

We remark that this example serves to demonstrate the practical applicability of our method. However, it is essential to note that while this example provides valuable insights, it may not be considered complete in terms of numerical implementations and comparisons with other methods. Further research and investigation would be necessary to address these aspects comprehensively.

  \section{Conclusion}\label{Conc}
  In this paper, we have introduced a stochastic data-driven Bouligand Landweber iteration (referred as SDBLI) method designed for tackling system of non-smooth inverse problems featuring  data-driven operators. Our approach is particularly suitable for scenarios where the Fr\'echet derivatives $F_{i}'$ of $F_i$ are not defined. We have demonstrated the feasibility of substituting the Bouligand subderivatives of $F_i$ in place of $F_{i}'$. Our convergence analysis   is based on a generalized tangential cone condition and a Lipschitz boundedness assumption, which effectively handles the discontinuities associated with the Bouligand subderivative.\\
Furthermore, we have delved into the examination of the asymptotic stability of our method. In order to substantiate the practical utility of our approach, we have furnished an illustrative example that meets all the prerequisites essential for our analysis.\\
Lastly, we highlight several avenues for future research. 
Firstly, it is crucial to validate our assumptions and analysis within the context of specific non-linear, non-smooth inverse problems, thus demonstrating the practicality of SDBLI. Secondly, as demonstrated in \cite{jahn2020discrepancy}, the applicability of the discrepancy principle as a stopping criterion for SGD is of significant practical interest. Therefore, it is essential to conduct an analysis of an a-posteriori stopping rule for our method. Thirdly, it is imperative to examine adaptive step size rules in order to offer practical guidance that is valuable in real-world applications.

\end{document}